\documentclass[11pt]{amsart}
\usepackage{amssymb, times}
\usepackage{graphicx}
\usepackage{amsfonts}
\usepackage{amsmath}
\usepackage{amsthm}
\usepackage{fancyhdr}
\usepackage{indentfirst}
\usepackage{hyperref}
\usepackage{comment}
\usepackage{color}
\usepackage{subcaption}
\usepackage{epsfig}
\usepackage{pst-grad} 
\usepackage{pst-plot} 
\usepackage[space]{grffile} 
\usepackage{etoolbox} 
\usepackage{mathrsfs} 
\usepackage{enumitem} 
\usepackage{cancel} 

\newcommand{\Ac}{\mathcal{A}^h}
\newcommand{\al}{\alpha}

\newcommand{\V}{\mathcal{V}}
\newcommand{\R}{\mathbb{R}}
\newcommand{\N}{\mathbb{N}}
\newcommand{\mH}{\mathcal{H}}
\newcommand{\F}{\mathcal{F}}

\newcommand{\sB}{\widetilde{B}}
\newcommand{\C}{\mathcal{C}}

\newcommand{\lc}{\scalebox{1.8}{$\llcorner$}}
\newcommand{\La}{\Lambda}
\newcommand{\si}{\sigma}
\newcommand{\Si}{\Sigma}
\newcommand{\de}{\delta}

\newcommand{\ep}{\epsilon}
\newcommand{\pr}{\prime}
\newcommand{\Om}{\Omega}

\newcommand{\Ga}{\Gamma}
\newcommand{\ga}{\gamma}

\newcommand{\lap}{\triangle}
\newcommand{\ti}{\tilde}

\newcommand{\Z}{\mathcal{Z}}
\newcommand{\mS}{\mathcal{S}}
\newcommand{\f}{\mathbf{f}}

\newcommand{\M}{\mathbf{M}}
\newcommand{\bL}{\mathbf{L}}
\newcommand{\bI}{\mathbf{I}}
\newcommand{\mf}{\mathbf{f}}

\newcommand{\mF}{\mathbf{F}}
\newcommand{\mR}{\mathcal{R}}

\newcommand{\X}{\mathfrak{X}}
\newcommand{\sA}{\mathscr{A}}

\newcommand{\btau}{\boldsymbol\tau}
\newcommand{\bleta}{\boldsymbol{\eta}}
\newcommand{\n}{\mathbf{n}}

\newcommand{\md}{\mathbf{d}}

\newcommand{\rom}[1]{\expandafter\romannumeral #1}
\newcommand{\Rom}[1]{\uppercase\expandafter{\romannumeral #1}}

\newcommand{\VarTan}{\operatorname{VarTan}}
\newcommand{\spt}{\operatorname{spt}}
\newcommand{\dist}{\operatorname{dist}}
\newcommand{\Div}{\operatorname{div}}

\newcommand{\vol}{\operatorname{Vol}} 

\newcommand{\Ric}{\operatorname{Ric}}
\newcommand{\Area}{\operatorname{Area}}

\newcommand{\Clos}{\operatorname{Clos}}
\newcommand{\interior}{\operatorname{int}}

\begin{document}

\newtheorem{theorem}{Theorem}[section]
\newtheorem{proposition}[theorem]{Proposition}
\newtheorem{corollary}[theorem]{Corollary}

\newtheorem{claim}{Claim}

\theoremstyle{remark}
\newtheorem{remark}[theorem]{Remark}

\theoremstyle{definition}
\newtheorem{definition}[theorem]{Definition}

\theoremstyle{plain}
\newtheorem{lemma}[theorem]{Lemma}

\numberwithin{equation}{section}

\title[Min-max theory for Prescribed mean curvature]{Existence of hypersurfaces with prescribed mean curvature I -- Generic min-max}

\author[Xin Zhou]{Xin Zhou}
\address{Department of Mathematics, University of California Santa Barbara, Santa Barbara, CA 93106, USA}
\email{zhou@math.ucsb.edu}

\author[Jonathan J. Zhu]{Jonathan J. Zhu}
\address{Department of Mathematics, Harvard University, Cambridge, MA 02138, USA}
\email{jjzhu@math.harvard.edu}

\maketitle

\pdfbookmark[0]{}{beg}

\begin{abstract}
We prove that, for a generic set of smooth prescription functions $h$ on a closed ambient manifold, there always exists a nontrivial, smooth, closed hypersurface of prescribed mean curvature $h$. The solution is either an embedded minimal hypersurface with integer multiplicity, or a non-minimal almost embedded hypersurface of multiplicity one. 

More precisely, we show that our previous min-max theory, developed for constant mean curvature hypersurfaces, can be extended to construct min-max prescribed mean curvature hypersurfaces for certain classes of prescription function, including smooth Morse functions and nonzero analytic functions. In particular we do not need to assume that $h$ has a sign.
\end{abstract}

\setcounter{section}{-1}

\section{Introduction}
\label{S:intro}

Given a function $h:M\rightarrow \mathbb{R}$ on an ambient manifold, a hypersurface $\Sigma\subset M$ has prescribed mean curvature $h$ if its mean curvature satisfies \begin{equation}H_\Sigma = h|_\Sigma.\end{equation} Prescribed mean curvature (PMC) hypersurfaces $\Sigma=\partial \Omega$ are critical points of the functional 
\begin{equation}
\label{E:Ac0}
\mathcal{A}^h = \Area - \vol_h, 
\end{equation}
where $\vol_h(\Om)=\int_\Om h$ is the enclosed $h$-volume. PMC hypersurfaces are a canonical generalization of minimal and constant mean curvature (CMC) hypersurfaces, and have applications to physical phenomena such as capillary surfaces \cite[\S 1.6]{Finn86}. Indeed, the local existence theory (or Dirichlet problem) for PMC hypersurfaces is quite well-understood, with several results extending naturally from the CMC to the PMC setting \cite{Hildebrandt69b, Gulliver-Spruck71, Gulliver-Spruck72, Gulliver73, Steffen76, Duzaar93}. The global theory or existence problem for closed PMC hypersurfaces, however, is to the authors' knowledge almost completely open for nonconstant prescription functions $h$. 

In this article, we initiate a program to resolve the existence of closed PMC hypersurfaces, based on extending our min-max theory developed in \cite{Zhou-Zhu17} for CMC hypersurfaces. In particular, we prove that there exists a closed hypersurface of PMC $h$ for a generic set of prescription functions $h$:

\begin{theorem}
\label{thm:mainA}
Let $M^{n+1}$ be a smooth, closed Riemannian manifold of dimension $3\leq n+1\leq 7$. There is an open dense set $\mathcal{S}\subset C^\infty(M)$ of prescription functions $h$, for which there exists a nontrivial, smooth, closed, almost embedded hypersurface $\Sigma^n$ of prescribed mean curvature $h$. 
\end{theorem} 

In subsequent works, we plan to complete the global existence theory by approximating an arbitrary given smooth function $h\in \mathcal{C}^\infty(M)$. Here a hypersurface is almost embedded if it locally decomposes into smooth sheets that touch but do not cross; in fact, our constructed PMC hypersurfaces $\Sigma$ decompose into at most two sheets, and have touching set of codimension 1 - this is somewhat surprisingly the same regularity as we obtained in the CMC setting \cite{Zhou-Zhu17}. The dimension restriction comes from the regularity theory for stable minimal hypersurfaces \cite{SS81} and is typical of variational methods for hypersurfaces.

The existence problem for closed PMC hypersurfaces may be viewed as a higher dimensional \textit{extended} Arnold-Novikov conjecture. The original question, which remains open, is the existence of closed embedded curves of prescribed constant geodesic curvature on a topological $S^2$ (see \cite{Ginzburg96, Rosenberg-Schneider11, Schneider11} for more backgrounds); Arnold \cite[1996-17, 1996-18]{Arnold04} later posed the natural extension to nonconstant prescribed curvature. Our previous theory \cite{Zhou-Zhu17} has already completely resolved the higher dimensional problem for constant (mean) curvature. 

In Theorem \ref{thm:mainA}, we use the set $\mathcal{S}=\mathcal{S}_{M,g}$ of Morse functions whose zero locus is a (possibly empty) hypersurface whose mean curvature vanishes to finite order. The mean curvature flow gives us a neat argument to show that this set is indeed generic:

\begin{proposition}
\label{prop:genA}
Let $(M,g)$ be a smooth closed Riemannian manifold. Consider the set $\mathcal{S}$ of smooth Morse functions $h$ such that the zero set $\{h=0\}=:\Sigma_0$ is a smooth closed hypersurface, and the mean curvature of ${\Sigma_0}$ vanishes to at most finite order. Then $\mathcal{S}$ is open and dense in $C^\infty(M)$. \end{proposition}

In fact, our min-max theory is powerful enough to handle still more general prescription functions, namely those that satisfy either 
\begin{enumerate}
\item[($\dagger$)] If $\Sigma^n\subset M^{n+1}$ is a smoothly embedded hypersurface and $h|_\Sigma$ vanishes to all orders at $p\in \Sigma$, then there exists $r>0$ for which $\Sigma_0= \{h=0\} \cap B_r(p)$ is a connected, smoothly embedded hypersurface tangent to $\Sigma$ at $p$, and if the mean curvature of $\Sigma_0$ vanishes to infinite order at any point, then it vanishes identically; or
\item[($\ddagger$)] The zero set of $h:M^{n+1}\rightarrow \mathbb{R}$ is contained in a countable union of connected, smoothly embedded $(n-1)$-dimensional submanifolds. 
\end{enumerate}

Any function in $\mathcal{S}$ satisfies $(\dagger)$, as does any analytic function on a real analytic manifold. A general function satisfying $(\dagger)$, however, could have nontrivial minimal hypersurfaces in its zero set. In this case each component of our constructed PMC hypersurface is either one of these (embedded) minimal components or a definitively non-minimal component, which is almost embedded with codimension 1 touching set. Note, however, that $(\ddagger)$ precludes the possibility of minimal PMC hypersurfaces.

\vspace{1em}

The local existence theory of PMC hypersurfaces is a natural problem for geometric PDE, and has been fairly well-understood by parametric methods, such as \cite{Hildebrandt69b, Gulliver-Spruck71, Gulliver-Spruck72, Gulliver73, Steffen76, Duzaar93}, and also non-parametric methods \cite{Giaquinta74, Gerhardt74, Giusti78, Amster-Mariani03}. For the latter, however, graphical PMC hypersurfaces have typically been studied only for prescription functions independent of the vertical direction. In \cite[Problem 59]{Yau82}, Yau posed the global existence problem of finding PMC surfaces in $\mathbb{R}^3$, which inspired some important partial results such as \cite{Treibergs-Wei83, Gerhardt98, Yau97}. Finally we take care to mention the results of \cite{Duzaar-Steffen96}, obtained through a homological (currents) approach.

\vspace{1em}
In order to prove Theorem \ref{thm:mainA}, we consider the variational properties of PMC hypersurfaces as critical points of the $\mathcal{A}^h$-functional. Because we are interested in the global problem of finding closed PMC hypersurfaces, it is far from clear direct minimization of the $\mathcal{A}^h$ functional can produce a solution other than the empty domain $\mathcal{A}^h(\emptyset)=0$ or the total manifold $\mathcal{A}^h(M) = -\int_M h$. Therefore, the min-max method becomes the natural way to find nontrivial critical points of $\Ac$. 

The min-max theory for minimal submanifolds was initiated by Almgren \cite{AF62}, and has been a remarkably successful tool in the study of the area functional. Using methods of geometric measure theory, Almgren was able to prove the existence of a nontrivial weak solution as stationary integral varifolds \cite{AF65} in any dimension and codimension. In codimension one, higher regularity was established by Pitts \cite{P81} (for $2\leq n\leq 5$), and later Schoen-Simon [37] (for $n\geq 6$). Colding-De Lellis \cite{CD03} established the corresponding theory using smooth sweepouts based on ideas of Simon-Smith \cite{Sm82}. The preceding body of work completely resolved the $h\equiv 0$ case of Theorem \ref{thm:mainA}.

Recently, Marques-Neves \cite{MN14, Agol-Marques-Neves16, MN17} used the Almgren-Pitts min-max theory to resolve a number of longstanding open problems in geometry, including their celebrated proof of the Willmore conjecture. Consequently, the min-max program has seen a series of developments in various directions, including (but not limited to) \cite{DT13, Guaraco15, Ketover-Zhou15, MN16, Liokumovich-Marques-Neves18, DR16, LiZ16, Irie-Marques-Neves18, Chodosh-Mantoulidis18, Marques-Neves18, Song18}. 
In \cite{Zhou-Zhu17}, we extended the min-max theory to study the $\mathcal{A}^c$ functional ($h\equiv c$) for the construction of CMC hypersurfaces. The present work thus represents a natural continuation of our work, further extending the min-max theory to the PMC setting and the study of the $\mathcal{A}^h$ functional. 

\subsection{Min-max procedure} 

We now give a heuristic overview of our min-max method. In the main proofs, for technical reasons we will work with discrete families as in Almgren-Pitts, but here we will describe the key ideas using continuous families to elucidate those ideas.

Let $M$, $h$ be as in Theorem \ref{thm:mainA}. The $\mathcal{A}^h$ functional (\ref{E:Ac0}) is defined on open sets $\Omega$ with rectifiable boundary by  $\mathcal{A}^h(\Omega) = \Area(\partial \Omega) -  \int_\Omega h$. Denote $I=[0, 1]$. Consider a continuous $1$-parameter family of sets with rectifiable boundary
\[ \{\Om_x: x\in I\}, \text{ with $\Om_0=\emptyset$ and $\Om_1=M$}. \]

Fix such a family $\{\Om^0_x\}$, and consider its homotopy class $[\{\Om^0_x\}]=\big{\{ } \{\Om_x\}\sim \{\Om^0_x\} \big{\}}$.
The {\em $h$-min-max value} (or {\em $h$-width)} is defined as
\[ \bL^h=\inf_{\{\Om_x\}\sim \{\Om^0_x\}}\max\{\Ac(\Om_x): x\in I\}. \]

A sequence $\{\{\Om^i_x\}: i\in\N\}$ with $\max_{x\in I} \Ac(\Om^i_x)\to \bL^h$ is typically called a {\em minimizing sequence}, and any sequence $\{\Om^i_{x_i}: x_i\in (0, 1), i\in N\}$ with $\Ac(\Om^i_{x_i})\to \bL^h$ is called a {\em min-max sequence}. 

Our main result (for the precise statement see Theorem \ref{thm:main-body}) then says that there is a nice minimizing sequence $\{\{\Om^i_x\}: i\in\N\}$, and some min-max sequence $\{\Om^i_{x_i}: x_i\in (0, 1), i\in N\}$, such that: 

\begin{theorem} 
\label{T: main}
Suppose that $h$ satisfies $(\dagger)$ or $(\ddagger)$ and that $\int_M h \geq 0$. Then the sequence $\partial\Om^i_{x_i}$ converges as varifolds to a nontrivial, smooth, closed, almost embedded hypersurface $\Si$ of prescribed mean curvature $h$. Each component of $\Sigma$ is either a closed, embedded minimal hypersurface in the zero set of $h$; or an almost embedded hypersurface on which the convergence is multiplicity 1. 

If $h$ satisfies $(\ddagger)$, only the latter case may occur.
\end{theorem}

Here, and for the remainder of this article, we make the assumption that $\int_M h\geq 0$. This ensures that $\mathcal{A}_h(M) \leq 0$ and hence any positive width sweepout must have a nontrivial maximal slice; we can always guarantee this condition by changing the sign of $h$, and doing so does not affect the existence result, since a two-sided hypersurface with PMC $h$ will have PMC $-h$ under its opposite orientation. 


Our proof broadly follows the Almgren-Pitts scheme, but with several important difficulties. An important observation is that several of the innovations we developed in \cite{Zhou-Zhu17} may be used to handle the $\mathcal{A}^h$ functional even for nonconstant $h$. However, in the present setting we must develop the background compactness and regularity theory for PMC hypersurfaces, and when $h$ is allowed to have a significant zero set it becomes crucial to obtain good control of any touching phenomena. 

To describe these challenges, we first review the Almgren-Pitts min-max method, which may be organized into five broad steps as follows:

\begin{itemize}
\item Construct a sweepout with positive width, and extract a minimizing sequence;
\item Apply a `tightening' map to construct a new sequence whose varifold limit satisfies a variational property and an `almost-minimizing' property;
\item Use these properties to construct `replacements' on annuli which must be regular;
\item Apply successive concentric annular replacements to the min-max limit and show that they coincide with each other, and hence extend to the center;
\item Show that the min-max limit coincides with the replacement near the center.
\end{itemize}

Given a minimizing sequence, in the minimal ($h\equiv 0$) setting one constructs a new `tightened' sequence, for which any min-max (varifold) limit must be stationary - that is, a weak solution in the sense of first variations. The $\mathcal{A}^h$ functional, however, is not well-defined on varifolds, so it is not straightforward to formulate a notion of weak solution for its critical points. In the PMC setting, we are able to show that the min-max limit $V$ (after tightening) has first variation bounded by $c= \sup |h|$. As in \cite{Zhou-Zhu17}, it is an important and delicate observation that this relatively loose variational property still provides enough control to develop the remaining regularity theory. 

We also show that the limit $V$ is {\em $h$-almost minimizing}, which we formulate as property of being (a limit of) constrained almost-minimizers for the $\mathcal{A}^h$ functional (see Definition \ref{D:c-am-varifolds}). To construct {(\em $h$-) replacements} in a subset $U\subset M$, one then solves the corresponding series of constrained minimization problems. For the $\mathcal{A}^h$ functional, each (local) $\mathcal{A}^h$-minimizer will be an open set $\Omega_i^*$ with stable, regular PMC boundary in $U$, and the $h$-replacement $V^*$ is obtained as the varifold limit $\lim |\partial \Omega_i^*|$. At this point, it is clear that a good compactness theory for PMC hypersurfaces is essential for the regularity of $V^*$.

To establish the regularity of $V$, successive replacements $V^*$ and $V^{**}$ are applied on two overlapping concentric annuli $A_1$ and $A_2$. The goal is to show that the replacements glue together smoothly on the overlapping region and may thus be extended all the way to the center, by taking further replacements. Here the main technical issue is to preserve the regularity and uniqueness across the gluing interface. 

Finally, one would like to prove that the min-max limit $V$ coincides with the extended replacement $V^*$ near the center. In the minimal setting, one can appeal to the constancy theorem. Since in the PMC setting we also make an assertion about the multiplicity of $V$, we instead directly prove that $V^*$ has at worst a removable singularity at the center, then use a moving sphere argument to show that the densities of $V$ and $V^*$ are the same in the annular region. 

\vspace{1em}

The primary concern in the PMC setting is the compactness and regularity theory for PMC hypersurfaces. Whilst it is not too difficult to show that limits of PMC hypersurfaces remain almost embedded, for the moving sphere and gluing arguments above it is essential that we may avoid the bulk of any touching singularities of the replacement hypersurfaces by choosing the gluing interface to be transverse to the touching set. If the touching sets are small enough - $(n-1)$-rectifiable, for instance - then we may indeed achieve this by Sard's lemma. 

Unlike the minimal or CMC settings, however, in the PMC setting we may not have a complete two-sided or even one-sided maximum principle. Moreover, whilst the difference of PMC hypersurfaces with opposite orientations will satisfy an elliptic PDE, there is an inhomogenous term which does not have a sign for general $h$. The danger that arises is the possibility of infinite order but non-identical touching, for which the tamest example is two PMC sheets sticking together as a minimal hypersurface on a small subset. To overcome this, we use unique continuation for elliptic differential \textit{inequalities} to prove the necessary compactness theorem: Under assumptions $(\dagger)$ or $(\ddagger)$, any infinite order touching for the limit implies that all sheets are identical, minimal and contained in the zero set of $h$. 

The possibility of minimal components also presents a technical issue for the gluing step, as again one would like to prevent non-minimal sheets gluing together into a higher multiplicity minimal sheet. We can rule this out by first obtaining a putative gluing to a smooth PMC hypersurface, then using unique continuation results coming from our work to control the touching set.

\vspace{1em}
To complete the details of the min-max procedure, we adapt several ideas that we previously introduced to handle the CMC setting \cite{Zhou-Zhu17}, demonstrating also the flexibility of those techniques. For instance, an essential observation is that the $h$-volume is still of lower order than the area term in the $\mathcal{A}^h$ functional. This immediately allows us to show that $\bL^c$ is positive on any sweepout, as a consequence of the isoperimetric inequality for small volumes (see Theorems \ref{T:Isoperimetric areas} and \ref{T:existence of nontrivial sweepouts}). Moreover, an important issue in both the CMC and PMC settings is that the total mass of the replacement $V^*$ may differ from the total mass of the original varifold $V$; however the mass defect is again controlled by the higher order term $\|h\|_\infty \vol(U)$, which converges to zero under any blowup process. Using this insight, we are able to prove that any blowup of the min-max limit $V$ has the good replacement property of Colding-De Lellis, and is therefore regular (see Proposition \ref{L:blowup is regular}); in particular the tangent cones of $V$ are always planes. 

Finally, to complete the gluing in the non-minimal setting, we must nevertheless work around a nontrivial touching set at the gluing interface. Again the compactness theory for PMC hypersurfaces is key, and an important step is to show that the second replacement $V^{**}$ may be represented by a boundary in $A_1\cup A_2$. This yields that $V^*$ and $V^{**}$ glue together - with matching orientations - as desired. Near the touching set, we then use the graphical decomposition into embedded sheets together with the gluing along regular part to properly match the sheets together.

\subsection{Outline of the paper}

Our basic notation and background material is described in Section \ref{S:Notation and background}. Then in Section \ref{S:Preliminaries} we recall some preliminary results including the regularity for $\mathcal{A}^h$-minimizers. 

In Section \ref{sec:compactness}, we describe curvature estimates and compactness for almost embedded PMC hypersurfaces, including unique continuation lemmas and estimates of the touching set for functions satisfying $(\dagger)$ or $(\ddagger)$. We also prove there the genericness of condition $(\dagger)$, that is, Proposition \ref{prop:genA}. 

In Section \ref{S:The c-Min-max construction} we formulate the precise min-max procedure, and prove the existence of nontrivial (positive width) sweepouts. %
Then in Section \ref{S:tightening} we review the tightening process for varifolds of bounded first variation and its consequences for the $\mathcal{A}^h$ functional. %
In Section \ref{S:c-Almost minimizing} we further observe that the replacement theory for constant prescribed mean curvature extends to an $h$-replacement theory for our suitable functions $h$. %

Finally, in Section \ref{S:Regularity for c-min-max varifold} we complete the regularity of the min-max varifold. 

\vspace{0.5em}
{\bf Acknowledgements}: 
X. Zhou is partially supported by NSF grant DMS-1811293.  
J. Zhu is partially supported by NSF grant DMS-1607871.
The authors want to thank Prof. Richard Schoen and Prof. Bill Minicozzi for their interest, advice and encouragement on this work. 
The second author would also like to thank Boyu Zhang for helpful discussions.

\section{Notation}
\label{S:Notation and background}

In this section, we collect some notions. We refer to \cite{Si83} and \cite[\S 2.1]{P81} for further materials in geometric measure theory.

Let $(M^{n+1}, g)$ denote a closed, oriented, smooth Riemannian manifold of dimension $3\leq (n+1)\leq 7$. Assume that $(M, g)$ is embedded in some $\R^L$, $L\in\N$. $B_r(p), \sB_r(p)$ denote respectively the Euclidean ball of $\R^L$ or the geodesic ball of $(M, g)$. We denote by $\mH^k$ the $k$-dimensional Hausdorff measure; $\bI_{k}(M)$ the space of $k$-dimensional integral currents in $\R^L$ with support in $M$; $\Z_{k}(M)$ the space of integral currents $T\in\bI_{k}(M)$ with $\partial T=0$; $\V_{k}(M)$ the closure, in the weak topology, of the space of $k$-dimensional rectifiable varifolds in $\R^L$ with support in $M$; $G_k(M)$ the Grassmannian bundle of un-oriented $k$-planes over $M$; $\F$ and $\M$ respectively the flat norm \cite[\S 31]{Si83} and mass norm \cite[26.4]{Si83} on $\bI_k(M)$; $\mF$ the varifold $\mF$-metric on $\V_k(M)$ and currents $\mF$-metric on $\bI_k(M)$, \cite[2.1(19)(20)]{P81}; $\C(M)$ or $\C(U)$ the space of sets $\Om\subset M$ or $\Om\subset U\subset M$ with finite perimeter (Caccioppoli sets), \cite[\S 14]{Si83}\cite[\S 1.6]{Gi}; and $\X(M)$ or $\X(U)$ the space of smooth vector fields in $M$ or supported in $U$.

We also utilize the following definitions:
\begin{enumerate}[label=\alph*), leftmargin=1cm]
\label{En: notations}
\item Given $T\in\bI_{k}(M)$, $|T|$ and $\|T\|$ denote respectively the integral varifold and Radon measure in $M$ associated with $T$;
\item Given $c>0$, a varifold $V\in \V_k(M)$ is said to have {\em $c$-bounded first variation in an open subset $U\subset M$}, if
\[ |\de V(X)|\leq c \int_M|X|d\mu_V, \quad \text{for any } X\in\X(U); \]
here the first variation of $V$ along $X$ is $\de V(X)=\int_{G_k(M)} div_S X(x)d V(x, S)$, \cite[\S 39]{Si83};
\item $U_r(V)$ denotes the ball in $\V_k(M)$ under $\mF$-metric with center $V\in\V_k(M)$ and radius $r>0$;
\item Given $p\in\spt\|V\|$, $\VarTan(V,p)$ denotes the space of tangent varifolds of $V$ at $p$, \cite[42.3]{Si83};
\item Given a smooth, immersed, closed, orientable hypersurface $\Si$ in $M$, or a set $\Om\in\C(M)$ with finite perimeter, $[[\Si]]$, $[[\Om]]$ denote the corresponding integral currents with the natural orientation, and $[\Si]$, $[\Om]$ denote the corresponding integer-multiplicity varifolds;
\item $\partial\Om$ denotes the (reduced)-boundary of $[[\Om]]$ as an integral current, and $\nu_{\partial\Om}$ denotes the outward pointing unit normal of $\partial \Om$, \cite[14.2]{Si83}.
\end{enumerate}

In this paper, we are interested in the following weighted area functional defined on $\C(M)$. Given $h:M\rightarrow \mathbb{R}$, define the {\em $\Ac$-functional} on $\C(M)$ as
\begin{equation}
\label{E: Ac}
\Ac(\Om)=\mH^n(\partial\Om)-\int_\Omega h. 
\end{equation}
The {\em first variation formula} for $\Ac$ along $X\in \X(M)$ is (see \cite[16.2]{Si83}) 
\begin{equation}
\label{E: 1st variation for Ac}
\de\Ac|_{\Om}(X)=\int_{\partial\Om}div_{\partial \Om}X d\mu_{\partial\Om}-\int_{\partial\Om}h\langle X,\nu\rangle \, d\mu_{\partial\Om},
\end{equation}
where $\nu= \nu_{\partial \Om}$ is the outward unit normal on $\partial \Om$. 

When the boundary $\partial\Om=\Si$ is a smooth immersed hypersurface, we have \[div_{\Si}X=H\langle   X,\nu\rangle,\] where $H$ is the mean curvature of $\Si$ with respect to $\nu$; if $\Om$ is a critical point of $\Ac$, then (\ref{E: 1st variation for Ac}) directly implies that $\Si=\partial \Om$ must have mean curvature $H=h|_\Sigma$ with respect to the outward unit normal $\nu$. In this case, we can calculate the {\em second variation formula} for $\Ac$ along normal vector fields $X\in \X(M)$ such that $X=\varphi\nu $ along $\partial\Om=\Si$ where $\varphi\in C^\infty(\Si)$, \cite[Proposition 2.5]{BCE88}, 
\begin{equation}
\label{E: 2nd variation for Ac}
\de^2\Ac|_{\Om}(X,X) = \Rom{2}_\Si (\varphi,\varphi) =\int_{\Si}\left( |\nabla\varphi|^2-\left(Ric^M(\nu, \nu)+|A^\Si|^2 - \partial_\nu h\right)\varphi^2\right)d\mu_{\Si}. 
\end{equation}
In the above formula, $\nabla\varphi$ is the gradient of $\varphi$ on $\Si$; $Ric^M$ is the Ricci curvature of $M$; $A^\Si$ is the second fundamental form of $\Si$.

\section{Preliminaries}
\label{S:Preliminaries}

In this section, we collect some preliminary results. We present a maximum principle for varifolds with bounded first variation, a regularity result for boundaries that minimize the $\Ac$-functional, and a result on isoperimetric profile for small volumes. 

\subsection{Maximum principle for varifolds with $c$-bounded first variation}

We will need the following maximum principle which is essentially due to White \cite[Theorem 5]{White10}. 

\begin{proposition}[Maximum principle for varifolds with $c$-bounded first variation]
\label{P:maximum principle}
Suppose $V\in \V_n(M)$ has $c$-bounded first variation in a open subset $U\subset M$. Let $K\subset U$ be an open subset with compact closure in $U$, such that $\spt(\|V\|)\subset K$, and
\begin{itemize}
\item[(i)] $\partial K$ is smoothly embedded in $M$,
\item[(ii)] the mean curvature of $\partial K$ with respect to the outward pointing normal is greater than $c$.
\end{itemize}
Then $\spt(\|V\|)\cap \partial K=\emptyset$. 
\end{proposition}

\subsection{Regularity for boundaries which minimize the $\Ac$ functional}

The following result about regularity of boundaries which minimize the $\Ac$ functional can be found in \cite{Mo03}.

\begin{theorem}
\label{T:regularity of Ac minimizers}
Given $\Om\in\C(M)$, $p\in\spt\|\partial\Om\|$, and some small $r>0$, suppose that $\Om\lc \sB_s(p)$ minimizes the $\Ac$-functional: that is, for any other $\La\in\C(M)$ with $\spt\|\La-\Om\|\subset \sB_s(p)$, we have $\Ac(\La)\geq \Ac(\Om)$. Then except for a set of Hausdorff dimension at most $n-7$, $\partial\Om\lc \sB_s(p)$ is a smooth and embedded hypersurface, and is real analytic if the ambient metric on $M$ is real analytic.
\end{theorem}
\begin{proof}

The regularity will follow from the arguments in \cite[Section 3]{Mo03} - in particular \cite[Corollary 3.7, 3.8]{Mo03} - so long as we can verify condition \cite[3.1(2)]{Mo03}. That is, setting $A = \mH^n(\partial\Omega \cap \sB_r(p))$ and $A' = \mH^n(\partial \La\cap \sB_r(p))$, 
it suffices to prove that 
\begin{equation}
\label{eq:morgan312}
A'-A\geq -Cr A, 
\end{equation}
for all $\La \in \C(M)$ as in the statement of the theorem and small enough $r$. Indeed, for such $\Lambda$, since $\Om\lc \sB_r(p)$ minimizes the $\Ac$-functional in $\sB_r(p)$ we have that 
\begin{equation}\label{eq:Ah-minimiser}A'-A \geq - \left| \int_\Lambda h - \int_\Omega h \right| \geq -c \mathcal{H}^{n+1}(\Lambda \triangle \Omega),\end{equation}
where $c=\sup |h|$ and $\Lambda \triangle \Omega$ is the symmetric difference.

In fact this essentially replaces condition \cite[3.1(1)]{Mo03}, and we finish the proof in the same way, by estimating 
$\mathcal{H}^{n+1}(\Lambda \triangle \Omega) \leq 2 \omega_{n+1} r^{n+1},$ where $\omega_n$ is the volume of the unit $n$-ball and, by the isoperimetric inequality, $\mH^{n+1}(\La \triangle \Omega)\leq C \mH^n(\partial(\La \triangle \Omega))^\frac{n+1}{n}\leq  C (A'+A)^\frac{n+1}{n}\nonumber$. Combining these estimates with (\ref{eq:Ah-minimiser}) yields that \[  A'-A\geq -Cr(A'+A), \]
which is clearly equivalent to (\ref{eq:morgan312}) for small $r$. This completes the proof. 
\end{proof}

\subsection{Isoperimetric profiles for small volume}

We will use the following consequence of the fact that the isoperimetric profile is asymptotically Euclidean for small volumes \cite{BM82} (see also \cite[Theorem 3]{Nardulli09}). Note that the result indeed holds for any $\Om\in\C(M)$ by using the regularity theory for isoperimetric domains (c.f. Theorem \ref{T:regularity of Ac minimizers}).
\begin{theorem}
\label{T:Isoperimetric areas}
There exists constants $C_0>0$ and $V_0>0$ depending only on $M$ such that 
\[ \Area(\partial\Om)\geq C_0 \vol(\Om)^{\frac{n}{n+1}}, \text{ whenever $\Om\in\C(M)$ and $\vol(\Om)\leq V_0$.}\]
\end{theorem}

\section{Stable PMC hypersurfaces}
\label{sec:compactness}

In this section, we establish curvature estimates and the compactness theory for stable hypersurfaces of prescribed mean curvature $h:M\rightarrow \mathbb{R}$. In particular, when $h$ satisfies either assumption $(\dagger)$ or $(\ddagger)$, we obtain good control on the touching set that arises upon taking the limit of embedded stable PMC hypersurfaces. 

\subsection{Stability and curvature estimates}

\begin{definition}
\label{D:stable c-hypersurface}
Let $\Si$ be a smooth, immersed, two-sided hypersurface with unit normal vector $\nu$, and $U\subset M$ an open subset. We say that $\Si$ is a {\em stable $h$-hypersurface} in $U$ if 
\begin{itemize}
\item  {the mean curvature $H$ of $\Si\cap U$ with respect to $\nu$ equals to $h|_\Sigma$; and}
\item $\Rom{2}_\Si(\varphi, \varphi)\geq 0$ { for all $\varphi\in C^\infty(\Si)$ with $\spt{\varphi}\subset \Si\cap U$,} 
where $\Rom{2}_\Si$ is as in (\ref{E: 2nd variation for Ac}).
\end{itemize}
\end{definition}

\begin{definition}
\label{def:comparisonsheets}
Let $\Sigma_i$, $i=1,2$, be connected embedded hypersurfaces in a connected open subset $U\subset M$, with $\partial\Sigma_i\cap U=\emptyset$ and unit normals $\nu_i$. We say that {\em $\Sigma_2$ lies on one side of $\Sigma_1$} if $\Sigma_1$ divides $U$ into two connected components $U_1 \cup U_2 = U\setminus \Sigma_1$, where $\nu_1$ points into $U_1$, and either:
\begin{itemize}
\item $\Sigma_2 \subset \Clos(U_1)$, which we write as $\Sigma_1\leq \Sigma_2$ or that $\Sigma_2$ lies on the positive side of $\Sigma_1$; or
\item $\Sigma_2 \subset \Clos(U_2)$, which we write as $\Sigma_1\geq \Sigma_2$ or that $\Sigma_2$ lies on the negative side of $\Sigma_1$.
\end{itemize}
 \end{definition}

\begin{definition}[Almost embedding]
\label{D:almost embedded boundary}
Let $U\subset M^{n+1}$ be an open subset, and $\Si^n$ be a smooth $n$-dimensional manifold. A smooth immersion $\phi: \Si\rightarrow U$ is said to be an {\em almost embedding} if at any point $p\in\phi(\Si)$ where $\Si$ fails to be embedded, there is a small neighborhood $W\subset U$ of $p$, such that 
\begin{itemize}
\item $\Si\cap \phi^{-1}(W)$ is a disjoint union of connected components $\cup_{i=1}^l \Si_i$;
\item $\phi(\Si_i)$ is an embedding for each $i=1,\cdots,l$; 
\item for each $i$, any other component $\phi(\Si_j)$, $j\neq i$, lies on one side of $\phi(\Si_i)$ in $W$.
\end{itemize}
We will simply denote $\phi(\Si)$ by $\Si$ and denote $\phi(\Si_i)$ by $\Si_i$. The subset of points in $\Si$ where $\Si$ fails to be embedded will be called the {\em touching set}, and denoted by $\mS(\Si)$. We will call $\Si\backslash\mS(\Si)$ the regular set, and denote it by $\mR(\Si)$.
\end{definition}
\begin{remark}
From the definition, the collection of components $\{\Si_i\}$ meet tangentially along $\mS(\Si)$.
\end{remark}

\begin{definition}[Almost embedded $h$-boundaries]
\begin{enumerate}
\item An almost embedded hypersurface $\Si\subset U$ is said to be {\em a boundary} if there is an open subset $\Om\in\C(U)$, such that $\Si$ is equal to the boundary $\partial\Om$ (in $U$) in the sense of currents;

\item The {\em outer unit normal} $\nu_\Si$ of $\Si$ is the choice of the unit normal of $\Si$ which points outside of $\Om$ along the regular part $\mR(\Si)$; 

\item $\Si$ is called {\em a stable $h$-boundary} if $\Si$ is a boundary as well as a stable immersed $h$-hypersurface. 
\end{enumerate}
\end{definition}

We have the following variant of the famous Schoen-Simon-Yau (for $2\leq n\leq 5$) \cite{SSY75} and Schoen-Simon ($n=6$) \cite{SS81} curvature estimates.
\begin{theorem}[Curvature estimates for stable $h$-hypersurfaces]
\label{T:curvature estimates and compactness}
Let $2 \leq n \leq 6$, and $U\subset M$ be an open subset. If $\Si\subset U$ is a smooth, immersed (almost embedded when $n=6$), two-sided, stable $h$-hypersurface in $U$ with $\partial \Si\cap U=\emptyset$, and $\Area(\Si)\leq C$, then there exists $C_1$ depending only on $n, M, c = \sup|h| , C$, such that
\[ |A^\Si|^2 (x) \leq \frac{C_1}{\dist^2_M(x,\partial U)} \quad \text{ for all $x \in \Si$}. \]
Moreover if $\Si_k\subset U$ is a sequence of smooth, immersed (almost embedded when $n=6$), two-sided, stable $h$-hypersurfaces in $U$ with $\partial \Si_k\cap U=\emptyset$ and $\sup_{k} \Area(\Sigma_k) < \infty$, 
then up to a subsequence, $\Sigma_k$ converges locally smoothly (possibly with multiplicity) to some stable $h$-hypersurface $\Sigma_\infty$ in $U$. 
\end{theorem}
\begin{proof}
The compactness statement follows in the standard way from the curvature estimates. The curvature estimates follow from standard blowup arguments together with the Bernstein Theorem \cite[Theorem 2]{SSY75} and \cite[Theorem 3]{SS81}, the key being that the blowup will be a stable minimal hypersurface, and when $n=6$, the blowup of a sequence of almost embedded $h$-hypersurfaces will be embedded by the classical maximum principle for embedded minimal hypersurfaces (c.f. \cite{Colding-Minicozzi}).
\end{proof}


\subsection{Some hypotheses on the zero set}

We consider smooth prescription functions $h:M^{n+1}\rightarrow \mathbb{R}$ satisfying the following property: 
\begin{enumerate}
\item[($\dagger$)] If $\Sigma^n\subset M^{n+1}$ is a smoothly embedded hypersurface and $h|_\Sigma$ vanishes to all orders at $p\in \Sigma$, then there exists $r>0$ for which $\Sigma_0= \{h=0\} \cap B_r(p)$ is a connected, smoothly embedded hypersurface tangent to $\Sigma$ at $p$. Moreover, if the mean curvature of $\Sigma_0$ vanishes to infinite order at any point, then it vanishes identically. 
\end{enumerate}

This property is satisfied by any nonzero analytic function on a real analytic Riemannian manifold: 

\begin{lemma}
Suppose that $M$ is a real analytic manifold with real analytic metric $g$. Then any nonzero analytic function $h:M\rightarrow \mathbb{R}$ satisfies $(\dagger)$. 
\end{lemma}
\begin{proof}
Let $p, \Sigma$ be as in $(\dagger)$. 

Now suppose $h$ is a nonzero analytic function. Near $p$, there exists a real analytic hypersurface $\Sigma_0$ that agrees with $\Sigma$ to infinite order at $p$. Then $h|_{\Sigma_0}$ is analytic and vanishes to all orders at $p$, so $h$ must vanish identically on $\Sigma_0$, and we can take $r$ small enough so that $\{h=0\}\cap B_r(p) = \Sigma_0\cap B_r(p)$. Similarly, $H_{\Sigma_0}$ is an analytic function on $\Sigma_0$, so if it vanishes to all orders at a point then it must vanish identically on $\Sigma_0$. 
\end{proof}

If $M$ is only smooth, we can still show that functions satisfying $(\dagger)$ are generic: 

\begin{proposition}
Let $(M,g)$ be a smooth closed Riemannian manifold. Consider the set $\mathcal{S}$ of smooth Morse functions $h$ such that the zero set $\{h=0\}=:\Sigma_0$ is a smooth closed hypersurface, and the mean curvature of ${\Sigma_0}$ vanishes to at most finite order. 

Then $\mathcal{S}$ is open and dense in $C^\infty(M)$. Moreover each function in $\mathcal{S}$ satisfies $(\dagger)$. 
\end{proposition}
\begin{proof}
Recall that the set of Morse functions is open and dense in $C^\infty(M)$. In fact, since nondegenerate critical points are isolated, the singular set $\{h=\nabla h=0\}$ of any Morse function $h$ has only finitely many points. It follows that the set $\mathcal{S}_0$ of Morse functions $h$ with empty singular set $\{h=\nabla h=0\} =\emptyset$ is also open and dense. By the implicit function theorem, having empty singular set is equivalent to the condition that $\Sigma_0 = \{h=0\}$ is a smooth, {embedded}, closed hypersurface. 

First we address the openness. Take $h\in \mathcal{S}$. Near $\Sigma_0$, the gradient is bounded below, $|\nabla h|\geq \delta>0$, so any small perturbation of $h$ will have zero set given by a smooth hypersurface close to $\Sigma_0$. 
Moreover, since the mean curvature of the level sets of $h$ are given by $H=\Div\left(\frac{\nabla h}{|\nabla h|}\right)$, any bound for the order of vanishing of $H$ will be preserved under smooth perturbation. 

To show that $\mathcal{S}$ is dense, it suffices to show, given $h\in \mathcal{S}_0$, that we can construct a smooth perturbation of $h$ whose zero set does not have mean curvature vanishing to infinite order. First, it is clear that we may perturb $h$ so that no component of $\Sigma_0$ has identically vanishing mean curvature, so we assume this without loss of generality. 
The idea is then to run the mean curvature flow $\Sigma_{0,t}$ starting from the zero set $\Sigma_0$, and construct a smooth deformation $h^t$ of $h$ which has the flowing hypersurface $\Sigma_{0,t}$ as its zero set. 

One way to perform this procedure is to extend the mean curvature vector of $\Sigma_{0,t}$, for a short time interval $[0,t_0]$, to a smooth time-dependent vector field $X_t$ supported in a fixed neighbourhood of $\Sigma_0$. We may then take $h^t$ to be the pullback of $h$ under the flow of $X_t$. Alternatively, we can get somewhat more control by running mean curvature flow on all nearby level sets, and interpolating as follows:

Denote $\mathcal{N}_\epsilon = \{|h| <\epsilon\} $. Since $|\nabla h|$ is nonzero on $\Sigma_0$, for sufficiently small $|s| <\epsilon$ the level sets $\Sigma_s=\{h=s\}$ are smooth, closed, {embedded}, hypersurfaces. Fix a smooth function $0\leq \phi\leq 1$ that is equal to 1 outside $\mathcal{N}_{\epsilon/3}$, and equal to 0 inside $\mathcal{N}_{\epsilon/4}$. 

We claim that for a short time interval $[0,t_0]$, the level set flow on $\mathcal{N}_{\epsilon/2}$ starting from $h$ admits a smooth solution. That is, there exists a smooth family of smooth functions $\tilde{h}^t$, such that the level sets $\Sigma_{s,t}:=\{\tilde{h}^t = s\}$ are given by the classical mean curvature flow starting from $\Sigma_s$. For small enough $\epsilon$ and $t_0$, we also have $\{|\tilde{h}^t| <\epsilon/5\}\subset \mathcal{N}_{\epsilon/4}$; in particular $h$ and $\tilde{h}^t$ have the same sign at each point of $\mathcal{N}_{\epsilon/3} \setminus \mathcal{N}_{\epsilon/4}$. We defer the proof to Section \ref{sec:mcf}. With the functions $\tilde{h}^t$ in hand we may define $h^t = \phi h + (1-\phi) \tilde{h}^t$. 

By either construction, we have a smooth family of smooth Morse functions $h^t$, with zero set $\{h^t=0\} = \Sigma_{0,t} $ as desired. We claim that at positive times the mean curvature of a hypersurface flowing by mean curvature cannot vanish to infinite order, unless it was a minimal hypersurface to begin with. Again the proof is deferred to Section \ref{sec:mcf}. But we already ensured that no component of $\Sigma_0$ was minimal, so by the claim, the mean curvature of $\Sigma_{0,t}$ can only vanish to at most finite order as desired. 

Finally, if $h\in\mathcal{S}$ and $p, \Sigma$ are as in $(\dagger)$, then $p\in \Sigma_0$ and we need only check that $T_p\Sigma_0 = T_p\Sigma$. But this is clear as otherwise the vanishing of $h|_{\Sigma}$ to all orders would force $p$ to be a degenerate critical point of $h$ on $M$. 
\end{proof}

We can also handle smooth prescription functions with small zero set. In particular we can consider smooth functions $h:M^{n+1}\rightarrow \mathbb{R}$ satisfying the property:
\begin{enumerate}
\item[($\ddagger$)] The zero set $\{h=0\}$ is contained in a countable union of connected, smoothly embedded $(n-1)$-dimensional submanifolds. 
\end{enumerate}

\subsection{Some results on mean curvature flow}
\label{sec:mcf}

Here we prove two results on mean curvature flow that were needed to show that the set $\mathcal{S}$ was a generic set. First we give a smooth short-time existence result for level set flows when everything is uniformly smooth.

\begin{proposition}
Suppose that $h:M\rightarrow \mathbb{R}$ is a smooth function and that $0<\delta \leq |\nabla h| \leq \delta^{-1}$ on $\mathcal{N}_{\epsilon}=\{|h| < \epsilon\}$. Further suppose that the level sets $\Sigma_s = \{h=s\}$ are smooth, closed, {embedded} hypersurfaces, which then have uniformly bounded curvature. Then there exists $t_0>0$ and a smooth family of smooth functions $\tilde{h}^t : \mathcal{N}_{\epsilon/2} \rightarrow \mathbb{R}$, $t\in [0,t_0]$, such that for each fixed $|s|\leq \epsilon/2$, the level sets $\Sigma_{s,t} = \{\tilde{h}^t=s\}$ are given by the classical mean curvature flow of $\Sigma_s$. 
\end{proposition}
\begin{proof}
First, it follows from Evans-Spruck \cite{ES91} that there is a unique weak solution $\tilde{h}^t$ of the level set flow with initial data $h$. By short-time existence and continuous dependence on the smooth initial data for the mean curvature flow, there exists a short time interval $t_0>0$ such that for any $|s| \leq \epsilon$, the mean curvature flow $\Sigma_{s,t}$ starting from $\Sigma_s$ exists and is smooth for $t\in [0,t_0]$. 

Therefore by \cite{ES91} again, the level sets $\{\tilde{h}^t=s\}$ must coincide with the classical flows $\Sigma_{s,t}$. In particular, for $t\in[0,t_0]$, the mean curvature flow defines a smooth family of smooth, surjective maps $F^t : \mathcal{N}_{\epsilon}\rightarrow \mathcal{N}^t_{\epsilon} $, where $\mathcal{N}^t_{\epsilon} := \{ |\tilde{h}^t| < \epsilon\}$. We claim that the $F^t$ are in fact a smooth family of smooth \textit{diffeomorphisms}. 

By the avoidance principle for mean curvature flow, each $F^t$ is certainly injective. In fact, by \cite{ES91} once more, the distance between distinct level sets is bounded by ${d(\Sigma_{s_1,t} , \Sigma_{s_2,t})}\geq {d(\Sigma_{s_1},\Sigma_{s_2})}$. This implies that $dF^t$ is uniformly nonsingular and the inverse function theorem then implies the claim. 

By taking $t_0$ smaller if needed we can ensure that $\mathcal{N}^t_\epsilon \supset \mathcal{N}_{\epsilon/2}$. Then $\tilde{h}^t$ is well-defined on $\mathcal{N}_{\epsilon/2}$ and satisfies $\tilde{h}^t = h\circ (F^t)^{-1}$, which completes the proof. 
\end{proof}

Second, we establish that at positive times, the mean curvature on a hypersurface flowing by mean curvature cannot vanish to infinite order unless it is minimal. 

\begin{proposition}
\label{lem:mcf}
Let $\{\Sigma^n_t\}_{0\leq t \leq t_0}$ be a smooth mean curvature flow of closed connected hypersurfaces in $M^{n+1}$. If $H_{\Sigma_t}$ vanishes to all orders at $p\in \Sigma_t$ for some $t>0$, then $\Sigma_0$ must have been a minimal hypersurface. 
\end{proposition}
\begin{proof}
For hypersurfaces under mean curvature flow, the mean curvature satisfies the parabolic PDE $(\partial_t - \lap_g) H = (|A|^2+\Ric(\nu,\nu)) H$, where $\lap_g$ is the hypersurface Laplacian. Since the flow is smooth on $[0,t_0]$, we have $|A|^2+|\Ric(\nu,\nu)|\leq C < \infty$. The conclusion follows from the spacelike strong unique continuation principle for parabolic PDE (see for instance \cite[Corollary 4.2.7]{vessella} or \cite{fernandez03, AV04}) and backwards uniqueness for mean curvature flow \cite{Ko16, LM17}.   
\end{proof}

\subsection{Estimates for the touching set}

Let $h:M^{n+1}\rightarrow \mathbb{R}$ be a smooth function. Suppose that $\Sigma^n_1, \Sigma^n_2$ are connected smoothly embedded hypersurfaces of prescribed mean curvature $h$ in a connected open subset $U\subset M$, which lie to one side of one another. We would like to show that for certain classes of prescription functions $h$, it follows that either $\Sigma_1 = \Sigma_2$ or the touching set $\Sigma_1\cap \Sigma_2$ is $(n-1)$-rectifiable. 

This section is devoted to proving the following theorem:

\begin{theorem}
\label{thm:touchingset}
Suppose that $h:M^{n+1}\rightarrow \mathbb{R}$ satisfies either property $(\dagger)$ or $(\ddagger)$.

Let $\Sigma_1, \Sigma_2$ be connected smoothly embedded hypersurfaces of prescribed mean curvature $h$ in a connected open subset $U\subset M$, which lie to one side of one another. Then either $\Sigma_1 = \Sigma_2$, or the touching set $\Sigma_1 \cap \Sigma_2$ is contained in a countable union of connected, smoothly embedded $(n-1)$-dimensional submanifolds. 
\end{theorem}
\begin{proof}

By a covering argument, it suffices to prove that the assertion holds on a neighbourhood of each touching point $p$. So let $p \in \Sigma_1 \cap \Sigma_2$, then we may take a chart $(\tilde{B}_r(p),g) \simeq (B^{n+1}_r(0), g')$ for which $T_p\Sigma_1 = T_p\Sigma_2 $ corresponds to the plane $P=\{x_{n+1}=0\}$, $\nu_1(p)$ corresponds to $e_{n+1}$ and the hypersurfaces $\Sigma_i$ may be written as graphs $x_{n+1}=u_i(x)$ of smooth functions $u_i$ over the plane $P$.

The function $u_1$ then satisfies the quasilinear elliptic PDE \begin{equation}\label{eq:graphpmc}\begin{split}&Hu_1 = h(x,u_1(x)), \\& Hu_2 = \begin{cases} h(x,u_2(x)), & \text{ if }\nu_2(p)=\nu_1(p) \\ -h(x,u_2(x)), & \text{ if }\nu_2(p) = - \nu_1(p).\end{cases}\end{split}\end{equation}
Here $H= Q+H_0$ is the mean curvature operator, and the zero order term $H_0$ is the mean curvature of $P$ in $(B_r(0),g')$.

We now divide into cases:

\subsubsection{Same orientation}

If $\nu_1(p)=\nu_2(p)$, then the $u_i$ satisfy the same quasilinear elliptic PDE $Hu = h(x,u(x))$, so since $h$ is smooth and $u_1(0)=u_2(0)$, the strong maximum principle of Serrin \cite{serrin} implies that $u_1 \equiv u_2$. The standard connectedness argument for unique continuation then yields:

\begin{lemma}
\label{lem:sameorientation}
Let $h:M^{n+1}\rightarrow \mathbb{R}$ be a smooth function. Suppose that $\Sigma_1, \Sigma_2$ are connected smoothly embedded hypersurfaces of prescribed mean curvature $h$ in a connected open subset $U\subset M$, which lie to one side of one another. Suppose that $p\in\Sigma_1\cap \Sigma_2$ and $\nu_1(p)=\nu_2(p)$. Then $\Sigma_1 \equiv \Sigma_2$. 
\end{lemma}

\subsubsection{Minimal sheets}

It will be useful to record what occurs if one sheet, say $\Sigma_2$, is in fact a minimal hypersurface, so that $h|_{\Sigma_2}=0$. In this case, irrespective of orientation, $u_2$ again satisfies $Hu_2 = \pm 0 = h(x,u_2(x))$. Applying the strong maximum principle of Serrin \cite{serrin} again we have: 

\begin{lemma}
\label{lem:minsheet}
Let $h:M^{n+1}\rightarrow \mathbb{R}$ be a smooth function. Suppose that $\Sigma_1, \Sigma_2$ are connected smoothly embedded hypersurfaces of prescribed mean curvature $h$ in a connected open subset $U\subset M$, which lie to one side of one another. Suppose that $p\in\Sigma_1\cap \Sigma_2$ and $\Sigma_2$ is minimal. Then $\Sigma_1 \equiv \Sigma_2$. 
\end{lemma}

\subsubsection{Opposite orientation}
\label{sec:opporientation}

If $\nu_1(p)= -\nu_2(p)$, then the difference $v:= u_1 - u_2$ satisfies an inhomogeneous linear elliptic PDE of the form 
\begin{equation}
\label{eq:pde-opp}
Lv = h(x,u_1(x)) + h (x, u_2(x)).
\end{equation}
 If $h$ does not vanish at $p$, then $Lv(p) \neq 0$, so the Hessian of $v$ at $p$ has rank at least 1. The implicit function theorem then implies that, on a possibly smaller neighbourhood of $p$, the touching set $\{v=Dv=0\}$ is contained in an $(n-1)$-dimensional submanifold; (for more details one may consult \cite[Lemma 2.8]{Zhou-Zhu17}). Thus we have shown: 

\begin{lemma}
Let $h:M^{n+1}\rightarrow \mathbb{R}$ be a smooth function. Suppose that $\Sigma_1, \Sigma_2$ are connected smoothly embedded hypersurfaces of prescribed mean curvature $h$ in a connected open subset $U\subset M$, which lie to one side of one another. Suppose that $p\in\Sigma_1\cap \Sigma_2$ and $\nu_1(p)= -\nu_2(p)$. Then $\Sigma_1 \cap  \Sigma_2 \setminus \{h=0\}$ is contained in a countable union of connected, smoothly embedded $(n-1)$-dimensional submanifolds. 
\end{lemma}

It remains to estimate the stationary touching set $\Sigma_1\cap \Sigma_2 \cap \{h=0\}$, so in the remainder of this section we assume $h(p)=0$. Of course, if $h$ satisfies $(\ddagger)$, that is, if $\{h=0\}$ is already contained in a countable union of connected, smoothly embedded $(n-1)$-dimensional submanifolds, then we are already done. 

First suppose that $\Sigma_1,\Sigma_2$ have a finite-order touching at $p$, that is, if $v=u_1-u_2$ vanishes to finite order at $x=0$. Then by the work of Hardt-Simon \cite{HS89} the touching set $\{u=Du=0\}$ is in fact, again on a possibly smaller neighbourhood of $p$, contained in a countable union of $(n-2)$-dimensional submanifolds. 

It is for the remaining case of infinite-order touching that we require property $(\dagger)$. Indeed, suppose that $v=u_1-u_2$ vanishes to infinite order at 0. Then by differentiating (\ref{eq:pde-opp}) we see that $h(x,u_i(x))$ vanishes to infinite order at $x=0$, that is, $h|_{\Sigma_i}$ must vanish to infinite order at $p$. 

Therefore by property $(\dagger)$, we may choose a possibly smaller ball $\tilde{B}_r(p)$ and a possibly new chart so that $\{h=0\} \cap \tilde{B}_r(p)$ is minimal and is identified with the set $P=\{x_{n+1}=0\} \subset B^{n+1}_r(0)$. The smooth functions $u_i$ then describe the height of $\Sigma_i$ above the zero set of $h$. In particular \[\{u_1=u_2\} \cap \{h=0\} \subset \{u_1=u_2=h=0\},\]
and they satisfy (\ref{eq:graphpmc}) with zero order term $H_0=0$. 

To proceed it is useful to record the following lemma which handles the case where either graph separately vanishes to infinite order:

\begin{lemma}
Let $h: (B_r^{n+1},g) \rightarrow \mathbb{R}$ be a smooth function satisfying $\{h=0\} = \{x_{n+1}=0\} =: \Sigma_0$, and that $\Sigma_0$ is minimal with respect to $g$. Suppose that $\Sigma$ is a hypersurface in $B_r$ of prescribed mean curvature $h$ and that $\Sigma$ is given by the graph of $u$ over $\Sigma_0$. If $u$ vanishes to infinite order at $x=0$, then there exists $\delta>0$ such that $u\equiv 0$ on $|x|<\delta$; that is, $\Sigma$ coincides with $\{h=0\}$ on $B_\delta$. 
\end{lemma}
\begin{proof}
Since $u$ vanishes to infinite order at $x=0$, the mean curvature operator $Qu= H_\Sigma - H_{\Sigma_0}$ will be a perturbation of the Laplacian near 0. On the other hand, by supposition we have $H_\Sigma - H_{\Sigma_0} = h|_\Sigma$. So using that $h$ is Lipschitz, for sufficiently small $\delta>0$ we will have $|L_0 u | \leq C|u|+\epsilon |Du|$ for $|x|<\delta$, where $L_0 = a^{ij} D_{ij}$ is a uniformly elliptic operator with smooth coefficients $|a^{ij} - \delta^{ij}| <\epsilon$. 
Strong unique continuation for elliptic differential inequalities (see for instance Aronszajn \cite{aronszajn}) then implies that $u\equiv 0$ for $|x|<\delta$. 
\end{proof}

We pause to record the following corollary: 

\begin{corollary}[Unique minimal continuation]
\label{cor:mincont}
Suppose that $h:M^{n+1}\rightarrow \mathbb{R}$ satisfies either property $(\dagger)$. Let $\Sigma$ be a connected smoothly embedded hypersurface of prescribed mean curvature $h$ in a connected open subset $U\subset M$. If $h$ vanishes on an open subset of $\Sigma$ then $h|_\Sigma \equiv 0$, that is, $\Sigma$ is a minimal hypersurface in $U$. 
\end{corollary}

(The corresponding result for $(\ddagger)$ is trivial as in this case $h$ cannot vanish on an open subset of $\Sigma$.)

Now by the lemma, if some $u_i$ vanishes to infinite order at $x=0$, then $\Sigma_i = \{h=0\}$ on a neighborhood of $p$, and is therefore a minimal hypersurface. Lemma \ref{lem:minsheet} then implies that the other sheets also coincide with $\{h=0\}$ near $p$. Again the usual connectedness argument for unique continuation then shows that the $\Sigma_i$ must all be the same minimal hypersurface in $U$. 

The only remaining case is in which each $u_i$ vanishes to finite order at $p$. In this case, by a result of B\"{a}r \cite{baer} (see also \cite{BBH}) the zero set of $u_i$  is locally contained in the union of countably many $(n-1)$-dimensional submanifolds. 

This concludes the proof of Theorem \ref{thm:touchingset}.
\end{proof}

For convenience we state the following obvious corollaries of Theorem \ref{thm:touchingset} for almost embedded hypersurfaces: 
\begin{proposition}[Touching sets for almost embedded $h$-hypersurfaces]
\label{P:smooth touching set}
If the metric on $U^{n+1}$ is smooth, and $h:U^{n+1}\rightarrow \mathbb{R}$ satisfies either property $(\dagger)$ or $(\ddagger)$ then for any almost embedded hypersurface $\Si^n \subset U$ of prescribed mean curvature $h$, the touching set $\mS(\Si)$ is contained in a countable union of connected, embedded $(n-1)$-dimensional submanifolds.

In particular, the regular set $\mR(\Sigma)$ is open and dense in $\Sigma$.
\end{proposition}

\begin{corollary}[Unique continuation for immersed $h$-hypersurfaces]
\label{cor:uniqcont}
Suppose that $h:M^{n+1}\rightarrow \mathbb{R}$ satisfies either property $(\dagger)$ or $(\ddagger)$. Let $\Si_1^n, \Si_2^n$ be two connected immersed hypersurfaces of prescribed mean curvature $h$ in a connected open subset $U$. If $\Si_1$ and $\Si_2$ coincide on a nonempty open neighbourhood $U'$, then $\Si_1 = \Si_2$. \end{corollary}



\begin{remark}
In the case that the metric on $M$ and the function $h$ are real analytic, we have the stronger statement that the touching set is a finite union of real analytic subvarieties $\bigcup_{k=0}^{n-1} S^k$ of respective dimension $k$. This follows from  \cite[Theorem 5.2.3]{KP92}, since in this setting the operator $L$ will have analytic coefficients, and hence the difference $v$ is also real analytic.
\end{remark}

\subsection{Compactness of stable PMC hypersurfaces}
\label{SS:Compactness of stable CMC hypersurfaces}

We are now in a position to prove our main compactness theorem. 
\begin{theorem}[Compactness theorem for almost embedded stable $c$-hypersurfaces]
\label{T:compactness}
Let $2\leq n\leq 6$. Suppose $\Si_k\subset U$ is a sequence of smooth, almost embedded, two-sided, stable $h^{(k)}$-hypersurfaces in $U$, with $\sup_{k} \Area(\Sigma_k) < \infty$. Further suppose that $h^{(k)}$ converges smoothly to a smooth function $h$ satisfying either property $(\dagger)$ or $(\ddagger)$. 

Then, up to a subsequence, $\{\Sigma_k\}$ converges locally smoothly (with multiplicity) to some almost embedded stable $h$-hypersurface $\Sigma_\infty$ in $U$. 

If additionally $\{\Si_k\}$ are all boundaries, then each connected component $\Sigma^i_\infty$ of $\Sigma_\infty$ is either:

\begin{enumerate}
\item minimal and smoothly embedded, or 
\item not minimal, with density $1$ along its regular part and $2$ along the touching set $\mS(\Si_\infty)$;  

\end{enumerate}

moreover, unless $\Sigma_i^\infty$ is minimal with multiplicity $l\geq 2$, we have that $\Sigma^\infty$ is locally a boundary - that is, for any point $p\in \Sigma^i_\infty$ there is a neighborhood $\sB_p$ of $p$ and an $\Omega \in \C(M)$ so that \[\Sigma_\infty \cap \sB_p =\Sigma^i_\infty \cap \sB_p = \partial \Omega \lc \sB_p.\] 

On the other hand, if $h^{(k)}\to h=0$, then up to a subsequence, $\{\Sigma_k\}$ converges locally smoothly (with multiplicity) to some smoothly embedded stable minimal hypersurface $\Sigma_\infty$ in $U$.
\end{theorem}
\begin{proof}[Proof of Theorem \ref{T:compactness}]

The locally smooth convergence results follow straightforwardly from Theorem \ref{T:curvature estimates and compactness} and the almost embedded assumption. One can also rule out sheets of the same orientation coming together, using the maximum principle 
Lemma \ref{lem:sameorientation} and the classical maximum principle for embedded minimal hypersurfaces (c.f. \cite{Colding-Minicozzi}) respectively. 

Now we consider the case where the $\Sigma_k$ are all boundaries. Namely, denote $\Si_k=\partial\Om_k$ for some $\Om_k\in \C(U)$. By standard compactness \cite[Theorem 6.3]{Si83}, a subsequence of the $\partial\Om_k$ converges weakly as currents to some $\partial\Om_\infty$ with $\Om_\infty\in\C(U)$. 

Take an arbitrary point $p\in\Sigma_\infty $. Suppose $p$ has density $l\geq 1$. Since we know the limit is almost embedded, we have an ordered graphical decomposition of $\Sigma_\infty$ in a neighborhood $\sB_p \subset U$ of $p$ given by $\bigcup_{i=1}^l \Sigma^i_\infty$, where each sheet has outward unit normal $\nu^i_\infty$ and all sheets touch at $p$. By the locally smooth convergence of $\Si_k$ to $\Si_\infty$, for $k$ large enough $\Si_k\cap \sB_p$ also has an ordered graphical decomposition $\bigcup_{i=1}^{l_k}\Si_k^i$, where $\Sigma^i_k \rightarrow \Sigma^i_\infty$. 

We claim that if there are two distinct sheets $\Sigma^i_\infty, \Sigma^j_\infty$ with the same orientation $\nu^i_\infty(p)=\nu^j_\infty(p)$, then there must be another sheet $\Sigma^m_\infty$ in between, $i<m<j$, with the opposite orientation; and hence, $\Sigma_\infty \cap \sB_p$ is a minimal hypersurface with multiplicity $l$. 

Indeed, suppose for the sake of contradiction that $\Sigma_\infty$ has two sheets with the same orientation and no sheet in between. Then the same is true for $\Sigma_k$ for large enough $k$, and we may assume without loss of generality that $i=1, j=2$. Then $\Sigma^1_k,\Sigma^2_k$ have unit normals $\nu^1_k,\nu^2_k$ pointing in the same direction. But then $\sB_p\backslash (\Si^1_k\cup\Si^2_k)$ has three connected components $U_0, U_1, U_2$ such that, counting orientation, $(\partial U_0)\lc \sB_p=\Si^1_k$, $(\partial U_1)\lc \sB_p=\Si^2_k-\Si^1_k$, and $(\partial U_2)\lc \sB_p=-\Si^2_k$. 

On the other hand, for each $i$ the Constancy Theorem \cite[Theorem 26.27]{Si83} applied to $\Om_k\lc U_i$ implies that $\Om_k\lc U_i$ is identical to either $\emptyset$ or $U_i$. That is, $\Omega_k \lc \sB_p = \sum_{i=0}^2 a_i U_i$, where each $a_i =0,1$. It is then easy to see that any choice of the $a_i$ will contradict the fact that, counting orientation, $\partial(\Om_k\lc \sB_p) \lc \sB_p= \Sigma_k \cap \sB_p = \Si^1_k+\Si^2_k$. 

Thus if there were the two sheets $\Sigma^i_\infty, \Sigma^j_\infty$ with the same orientation, then there must be another sheet $\Sigma^m_\infty$ in between with the opposite orientation. But since all sheets touch at $p$, by Lemma \ref{lem:sameorientation} the sheets $\Sigma^i_\infty, \Sigma^j_\infty$ must coincide, which forces $\Sigma^m_\infty$ to coincide as well. 

But then $\Sigma^i_\infty = \Sigma^m_\infty$ has prescribed mean curvature $h$ with respect to both orientations, so $h$ must vanish and it must be a minimal hypersurface. Lemma \ref{lem:minsheet} then implies that any other sheet touching at $p$ must coincide with $\Sigma^i_\infty$ as claimed. 

Of course, if $l>2$ then there must be two sheets converging with the same orientation. Thus we have shown that either: $p$ has density $l=1$ and is a regular point; $p$ has density $l=2$ and is a touching point between two non-minimal sheets of opposite orientation; or $p$ is a regular point on a minimal hypersurface of density $l \geq 2$. 

In particular if $\Sigma^i_\infty$ is not minimal then it has density 1 on its regular part and density 2 on its touching set. Having density 1 on the regular part is enough to conclude that $\Sigma_\infty \cap \sB_p = \Sigma^i_\infty \cap \sB_p = \partial\Omega_\infty \lc \sB_p$ as currents.
\end{proof}


\section{The $h$-Min-max construction}
\label{S:The c-Min-max construction}

In this section, we present the setup of the min-max construction mainly following Pitts \cite{P81}. We also prove the existence of a non-trivial sweepout with positive $\Ac$-min-max value.

\subsection{Homotopy sequences.}\label{homotopy sequences}

We will introduce the min-max construction using the scheme developed by Almgren and Pitts \cite{AF62, AF65, P81}; see also \cite[Section 3]{Zhou-Zhu17}.

\begin{definition}
\label{cell complex} 
(cell complex).
\begin{enumerate}
\item Denote $I=[0, 1]$, $I_0=\partial I=I\backslash (0, 1)$;

\item For $j\in\N$, $I(1, j)$ is the cell complex of $I$, whose $1$-cells are all intervals of form $[\frac{i}{3^{j}}, \frac{i+1}{3^{j}}]$, and $0$-cells are all points $[\frac{i}{3^{j}}]$; 

\item For $p=0 ,1$, $\al\in I(1, j)$ is a $p$-cell if $dim(\al)=p$. A $0$-cell is also called a vertex;

\item $I(1, j)_p$ denotes the set of all $p$-cells in $I(1, j)$, and $I_0(1, j)_0$ denotes the set $\{[0], [1]\}$;

\item Given a $1$-cell $\al\in I(1, j)_1$, and $k\in\N$, $\al(k)$ denotes the $1$-dimensional sub-complex of $I(1, j+k)$ formed by all cells contained in $\al$. For $q=0, 1$, $\al(k)_q$ and $\al_0(k)_q$ denote respectively the set of all $q$-cells of $I(1, j+k)$ contained in $\al$, or in the boundary of $\al$;

\item The boundary homeomorphism $\partial: I(1, j)\rightarrow I(1, j)$ is given by $\partial[a, b]=[b]-[a]$ if $[a, b]\in I(1, j)_1$, and $\partial[a]=0$ if $[a]\in I(1, j)_0$;

\item The distance function $\md: I(1, j)_0\times I(1, j)_0\rightarrow\N$ is defined as $\md(x, y)=3^{j}|x-y|$;

\item The map $\n(i, j): I(1, i)_{0}\to I(1, j)_{0}$ is defined as: $\n(i, j)(x)\in I(1, j)_{0}$ is the unique element of $I(1, j)_0$, such that $\md\big(x, \n(i, j)(x)\big)=\inf\big\{\md(x, y): y\in I(1, j)_{0}\big\}$.
\end{enumerate}
\end{definition}

For a map to the space of Caccioppoli sets: $\phi: I(1, j)_{0}\rightarrow\C(M)$, the \emph{fineness} of $\phi$ is defined as:
\begin{equation}\label{fineness}
\mf(\phi)=\sup\Big\{\frac{\M\big(\partial\phi(x)-\partial\phi(y)\big)}{\md(x, y)}:\ x, y\in I(1, j)_{0}, x\neq y\Big\}.
\end{equation}
Similarly we can define the fineness of $\phi$ with respect to the $\F$-norm and $\mF$-metric. We use $\phi: I(1, j)_{0}\rightarrow\big(\C(M), \{0\}\big)$ to denote a map such that $\phi\big(I(1, j)_{0}\big)\subset\C(M)$ and $\partial\phi|_{I_{0}(1, j)_{0}}=0$, i.e. $\phi([0]), \phi([1])=\emptyset$ or $M$.

\begin{definition}\label{homotopy for maps}
Given $\de>0$ and $\phi_{i}: I(1, k_{i})_{0}\rightarrow\big(\C(M), \{0\}\big)$, $i=0,1$, we say \emph{$\phi_{1}$ is $1$-homotopic to $\phi_{2}$ in $\big(\C(M), \{0\}\big)$ with fineness $\de$}, if $\exists\ k_{3}\in\N$, $k_{3}\geq\max\{k_{1}, k_{2}\}$, and
$$\psi: I(1, k_{3})_{0}\times I(1, k_{3})_{0}\rightarrow \C(M),$$
such that
\begin{itemize}
\setlength{\itemindent}{1em}
\item $\mf(\psi)\leq \de$;
\item $\psi([i], x)=\phi_{i}\big(\n(k_{3}, k_{i})(x)\big)$, $i=0,1$;
\item $\partial\psi\big(I(1, k_{3})_{0}\times I_{0}(1, k_{3})_{0}\big)=0$.
\end{itemize}
\end{definition}

\begin{definition}\label{(1, M) homotopy sequence}
A \emph{$(1, \M)$-homotopy sequence of mappings into $\big(\C(M), \{0\}\big)$} is a sequence of mappings $\{\phi_{i}\}_{i\in\N}$,
$$\phi_{i}: I(1, k_{i})_{0}\rightarrow\big(\C(M), \{0\}\big),$$
such that $\phi_{i}$ is $1$-homotopic to $\phi_{i+1}$ in $\big(\C(M), \{0\}\big)$ with fineness $\de_{i}$, and
\begin{itemize}
\setlength{\itemindent}{1em}
\item $\lim_{i\rightarrow\infty}\de_{i}=0$;
\item $\sup_{i}\big\{\M(\partial\phi_{i}(x)):\ x\in I(1, k_{i})_{0}\big\}<+\infty$.
\end{itemize}
\end{definition}
\begin{remark}
Note that the second condition implies that $\sup_{i}\big\{\Ac(\phi_{i}(x)):\ x\in I(1, k_{i})_{0}\big\}<+\infty$.
\end{remark}

\begin{definition}\label{homotopy for sequences}
Given two $(1, \M)$-homotopy sequences of mappings $S_{1}=\{\phi^{1}_{i}\}_{i\in\N}$ and $S_{2}=\{\phi^{2}_{i}\}_{i\in\N}$ into $\big(\C(M), \{0\}\big)$, \emph{$S_{1}$ is homotopic to $S_{2}$} if $\exists\ \{\de_{i}\}_{i\in\N}$, such that
\begin{itemize}
\setlength{\itemindent}{1em}
\item $\phi^{1}_{i}$ is $1$-homotopic to $\phi^{2}_{i}$ in $\big(\C(M), \{0\}\big)$ with fineness $\de_{i}$;
\item $\lim_{i\rightarrow \infty}\de_{i}=0$.
\end{itemize}
\end{definition}

It is easy to see that the relation ``is homotopic to" is an equivalence relation on the space of $(1, \M)$-homotopy sequences of mappings into $\big(\C(M), \{0\}\big)$. An equivalence class is a \emph{$(1, \M)$-homotopy class of mappings into $\big(\C(M), \{0\}\big)$}. Denote the set of all equivalence classes by $\pi^{\#}_{1}\big(\C(M, \M), \{0\}\big)$.

\subsection{Min-max construction.}

\begin{definition}
(Min-max definition) Given $\Pi\in\pi^{\#}_{1}\big(\C(M, \M), \{0\}\big)$, define $\bL^h: \Pi\rightarrow\R^{+}$ to be the function given by:
\[ \bL^h(S)=\bL^h(\{\phi_{i}\}_{i\in\N})=\limsup_{i\rightarrow\infty}\max\big\{\Ac\big(\phi_{i}(x)\big):\ x \textrm{ lies in the domain of $\phi_{i}$}\big\}. \]
The \emph{$\mathcal{A}^h$-min-max value of $\Pi$} is defined as
\begin{equation}\label{width}
\bL^h(\Pi)=\inf\{\bL^h(S):\ S\in\Pi\}.
\end{equation}

A sequence $S=\{\phi_i\}\in\Pi$ is called a \emph{critical sequence} if $\bL^h(S)=\bL^h(\Pi)$. 

Given a critical sequence $S$, then $K(S)=\{V=\lim_{j\to \infty}|\partial \phi_{i_{j}}(x_{j})|:\ \textrm{$x_{j}$ lies in the domain of $\phi_{i_{j}}$}\}$ is a compact subset of $\V_n(M^{n+1})$. 
The \emph{critical set} of $S$ is the subset $C(S)\subset K(S)$ defined by 
\begin{equation}\label{eq:criticalset} C(S)=\{V=\lim_{j\rightarrow\infty}|\partial\phi_{i_j}(x_j)|:\ \text{with} \lim_{j\to\infty} \Ac(\phi_{i_j}(x_j))=\bL^h(S)\}; \end{equation}
we call any such sequence $\{\partial\phi_{i_j}(x_j)\}$ as in (\ref{eq:criticalset}) a \emph{min-max sequence}.
\end{definition}

Note that by \cite[4.1(4)]{P81}, we immediately have:
\begin{lemma}
\label{L:critical sequence}
Given any $\Pi\in\pi^{\#}_{1}\big(\C(M, \M), \{0\}\big)$, there exists a critical sequence $S\in \Pi$.
\end{lemma}

The main theorem of this paper is as follows:
\begin{theorem}
\label{thm:main-body}
Let $2\leq n\leq 6$. Consider a smooth closed Riemannian manifold $M^{n+1}$ and a smooth function $h$, which satisfies $\int_M h \geq 0$ as well as either property $(\dagger)$ or $(\ddagger)$. There exists $\Pi\in\pi^{\#}_{1}\big(\C(M, \M), \{0\}\big)$ and a critical sequence $S\in \Pi$ such that:
\begin{itemize}
\item $\bL^h(\Pi)=\bL^h(S)>0$;
\item There exists an element $V$ of $C(S)$ induced by a nontrivial, smooth, almost embedded, closed hypersurface $\Sigma^n \subset M$ of prescribed mean curvature $h$. Moreover $V$ has multiplicity one, except possibly on components of $\Sigma$ on which $h$ vanishes. 
\end{itemize}
\end{theorem}

\begin{proof}[Proof of Theorem \ref{thm:main-body}]
This follows from combining Theorem \ref{T:existence of nontrivial sweepouts}, Theorem \ref{T: existence of almost minimizing varifold} and Theorem \ref{T:main-regularity}.
\end{proof}

\subsection{Existence of nontrivial sweepouts}
\label{SS:Existence of nontrivial sweepouts}

\begin{theorem}
\label{T:existence of nontrivial sweepouts}
Given $c$, there exists $\Pi \in \pi^{\sharp}_1\big(\C(M, \M), \{0\}\big)$, such that for any function $h:M\rightarrow \mathbb{R}$ with $\sup |h| =c <\infty$ and $\int_M h\geq 0$, we have $\bL^h(\Pi)>0$. Moreover, for any critical sequence $S\in \Pi$, the critical set $C(S)$ is non-empty.
\end{theorem}
%
\begin{remark}
\label{rmk:lower bound of Ac}
We first describe a heuristic argument using smooth sweepouts which will help to reveal the key idea. Let $C_0>0$ and $V_0>0$ to be the constants in Theorem \ref{T:Isoperimetric areas}, and fix $0<V\leq V_0$ such that $V^{\frac{-1}{n+1}} >2c/C_0$. Note that $V$ only depends on $c,C_0,V_0$. Then for any $\Om$ with $\vol(\Om)=V$, we have
\begin{equation}
\label{E:lower bound of Ac}
\Ac(\Om) \geq C_0 V^{\frac{n}{n+1}} - cV > cV >0.
\end{equation}

Now consider any smooth 1-parameter family $\{\Om_x: x\in [0, 1]\}$ satisfying $\Om_0=\emptyset$ and $\Om_1=M$. Since $\{\Om_x\}$ sweeps out $M$, there must exist some $x_0\in(0, 1)$ such that $\vol(\Om_{x_0})=V$, whence $\max_{x\in [0, 1]}\Ac(\Om_x) \geq cV >0$. Since this holds for any sweepout we then have $\bL^h(\Pi) \geq cV>0$. 
\end{remark}

\begin{proof}[Proof of Theorem \ref{T:existence of nontrivial sweepouts}]
Note that $\left| \int_\Omega h \right| \leq c\vol(\Omega)$. The proof for the existence of $\Pi \in \pi^{\sharp}_1\big(\C(M, \M), \{0\}\big)$ with $\bL^h(\Pi)>0$ proceeds as in the CMC setting \cite[Theorem 3.9]{Zhou-Zhu17}.

The remaining assertion that the critical sets are nontrivial follows from the following lemma.
\begin{lemma}
Given $\bL>0$, then 
\[ \inf\{\M(\partial\Om): \Om\in\C(M) \text{ and } \Ac(\Om)\geq \bL\}>0. \]
\end{lemma}
\textit{Proof}:
Suppose this is not true, then there exists a sequence $\{\Om_i\}\subset \C(M)$, with 
\begin{equation}
\label{E:lower bound of Ah}
\Ac(\Om_i)=\M(\partial\Om_i)-\int_{\Om_i}h\geq \bL>0,
\end{equation} 
but
\[ \lim_{i\to\infty}\M(\partial\Om_i)=0. \]
Up to a subsequence, we can assume that $\lim_{i\to\infty}\Om_i=\Om_\infty$ weakly in the sense of Caccioppoli sets for some $\Om_\infty\in\C(M)$. In particular, the characteristic functions $\chi_{\Om_i}\to \chi_{\Om_\infty}$ in $L^1(M)$. By the lower semi-continuity, we have
\[ \M(\partial\Om_\infty)\leq \lim_{i\to\infty}\M(\partial\Om_i)=0. \]
By the Constancy Theorem, we must have $\Om_\infty=\emptyset$ or $\Om_\infty=M$.

Since we assumed that $\int_Mh\geq 0$, either case then yields a contradiction upon taking $i\to\infty$ in (\ref{E:lower bound of Ah}). This completes the proof of the lemma and hence the proof of Theorem \ref{T:existence of nontrivial sweepouts}.
\end{proof}

\section{Tightening}
\label{S:tightening}

In this section, we recall the tightening process, which we constructed to study the $\mathcal{A}^c$ functional (defined like $\mathcal{A}^h$ but with the constant function $c$) as part of our min-max theory for constant mean curvature surfaces \cite[\S 4]{Zhou-Zhu17}. The same process can be applied to the $\Ac$ functional, and we will prove that after applying the tightening map to a critical sequence, every element in the critical set has uniformly bounded first variation. Note that the tightening process is adapted from those in \cite[\S 4]{CD03} and \cite[\S 4.3]{P81}.

\subsection{Review of constructions in \cite[\S 4]{Zhou-Zhu17}}. 
\label{SS: Review of constructions in ZZ}

We fist recall several key ingredients obtained in \cite[\S 4]{Zhou-Zhu17} for the tightening process. In this section, we always let $c\equiv \sup_M|h|$. Given $L>0$, consider the set of varifolds in $\V_n(M)$ with $2L$-bounded mass: $A^L=\{V\in \V_n(M):\ \|V\|(M)\leq 2L\}$. 
Denote
\[ A^c_\infty=\big\{V\in A^L: |\de V(X)|\leq c \int|X| d\mu_V, \text{ for any }X\in\X(M) \big\}. \]
Given $V\in A^L$, we denote
\[ \ga=\ga(V)= \mF(V, A^c_\infty). \]
Given $X\in \X(M)$, we use $\Phi_X: \R^+\times M\rightarrow M$ to denote the one parameter group of diffeomorphisms generated by $X$.

We have constructed in \cite[\S 4]{Zhou-Zhu17}:
\begin{enumerate}[label=(\roman*)]
\item a map $X: A^L \rightarrow \X(M)$, which is continuous with respect to the $C^1$ topology on $\X(M)$;
\item two continuous functions $g:\R^+\rightarrow \R^+$ and $\rho: \R^+\rightarrow \R^+$, such that $\rho(0)=0$ and
\[ \de W(X(V))+c\int |X(V)|d\mu_W \leq -g\big(\ga(V)\big),\quad \text{ whenever } W\in A^L, \mF(W, V)\leq \rho\big(\ga(V)\big); \]
\item a continuous time function $T: [0, \infty) \rightarrow [0, \infty)$, such that 
         \begin{itemize}
         \item $\lim_{t\rightarrow 0}T(t)=0$, and $T(t)>0$ if $t\neq 0$;

         \item For any $V\in A^L$, denote $\Phi_V=\Phi_{X(V)}$, and $V_t=\big(\Phi_V(t)\big)_{\#}V$, then $\mF(V_t, V)<\rho(\ga)$ for all $0\leq t\leq T(\ga)$.
         \end{itemize}
\end{enumerate}

Our goal is to show that given $\Om\in\C(M)$ with $\partial\Om\in A^L$, we can deform $\Om$ by $\Phi_{|\partial \Om|}(t)$ to get a 1-parameter family $\Om_t=\Phi_{|\partial\Om|}(t)(\Om) \in\C(M)$, so that the $\Ac$ functional of $\Om_t$ for some $t>0$ can be deformed down by a fixed amount depending only on $\ga(|\partial\Om|)=\mF(|\partial\Om|, A^c_\infty)$. 

In particular, given $V\in A^L\backslash A^c_\infty$, denote
\[ \Psi_V(t, \cdot)=\Phi_V\big(T(\ga)t, \cdot\big),\quad \textrm{for } t\in[0, 1], \]
and $L: \R^+\rightarrow \R^+$, with $L(\ga)=T(\ga)g(\ga)$; then $L(0)=0$ and $L(\ga)>0$ if $\ga>0$. We can deform $V$ through the family  $\left\{V_t=\big(\Psi_V(t)\big)_{\#}V: t\in [0, 1]\right\} \subset U_{\rho(\gamma)}(V)$, so that
\begin{enumerate}
\item The map $(t, V)\rightarrow V_t$ is continuous under the $\mF$-metric;

\item When $V=\partial \Om$, $\Om\in\C(M)$, $\ga=\mF(|\partial\Om|, A^c_\infty)>0$, we have by (\ref{E: 1st variation for Ac})
\begin{equation}
\label{E: decrease Ac by isotopy}
\Ac(\Om_1)-\Ac(\Om) \leq \int_{0}^{T(\ga)}[\de \Ac|_{\Om_t}] (X(|\partial\Om)|) dt\leq -T(\ga)g(\ga)= -L(\ga)<0.
\end{equation}
(Note that $\de \Ac_{\Om}(X)\leq \de|\partial\Om|(X)+c\int|X|d\mu_{\partial\Om}$ by (\ref{E: 1st variation for Ac}).)
\end{enumerate}
Finally note that the flow $\Psi_V(t, \cdot)$ is generated by the vector field 
\begin{equation}
\label{E:tiX}
\ti{X}(V)=T(\ga)X(V).
\end{equation}

\subsection{Deforming sweepouts by the tightening map}
\label{SS:Deforming sweepouts by the tightening map}

We now apply our tightening map to the critical sequence provided by Lemma \ref{L:critical sequence}. As in our min-max theory for the CMC setting, the conclusion is that varifolds in the critical set have $c$-bounded first variation, where now $c= \sup |h|$. Indeed, the proof proceeds essentially unchanged, with the $\mathcal{A}^c$ functional replaced by $\mathcal{A}^h$. %

\begin{proposition}[Tightening]
\label{P:tightening}
Let $\Pi \in \pi^{\sharp}_1\big(\C(M, \M), \{0\}\big)$, and assume $\bL^h(\Pi)>0$. For any critical sequence $S^*$ for $\Pi$, there exists another critical sequence $S$ for $\Pi$ such that $C(S) \subset C(S^*)$ and each $V\in C(S)$ has $c$-bounded first variation, $c=\sup |h|$.
\end{proposition}

\begin{proof}

Take $S^*=\{\phi_i^*\}$, where $\phi^*_i: I(1, k_i)_{0}\to\big(\C(M), \{0\}\big)$, and $\phi^*_i$ is $1$-homotopic to $\phi^*_{i+1}$ in $\big(\C(M), \{0\}\big)$ with fineness $\de_i \searrow 0$. Let $\Xi_i: I(1, k_i)_0\times [0, 1]\to \C(M)$ be defined as
\[ \Xi_i(x, t)= \Psi_{|\partial \phi^*_i(x)|}(t) \big( \phi^*_i(x)\big) . \]

Denote $\phi_i^t(\cdot)=\Xi_i(\cdot, t)$. First, we first recall the following fact, the proof of which follows by a standard argument using (\ref{E: decrease Ac by isotopy}); (see \cite[\S 4.4 Claim 1]{Zhou-Zhu17}):

\vspace{0.3cm}
\textit{Claim 1: if $\lim_{i\to\infty}\Ac(\phi^1_i(x_i))=\bL^h(\Pi)$, then (up to relabeling) there is a subsequence $\{\phi^1_i(x_i)\}$ converging (as varifolds) to a varifold in $C(S^*)$ of $c$-bounded first variation.}


\vspace{0.3cm}
\textit{Proof of Claim 1:} By (\ref{E: decrease Ac by isotopy}),
\begin{equation}
\Ac(\phi_i^1(x_i))-\Ac(\phi^*_i(x_i)) = -L(\ga_i),
\end{equation}
where $\ga_i=\mF(|\partial \phi^*_i(x_i)|, A^c_\infty)$. Therefore,
\[ \bL^h(\Pi) = \lim \Ac(\phi_i^1(x_i)) = \lim \Ac(\phi^*_i(x_i)) - L(\lim \ga_i) \leq \bL^h(\Pi) - L(\lim\ga_i), \] 
so actually we must have $\lim \ga_i=0$ and this implies that $\lim|\partial \phi^*_i(x_i)| \in A^c_\infty$. Moreover, by our construction of the tightening map (see property (iii) in \S \ref{SS: Review of constructions in ZZ}), each $|\partial\phi^1_i(x_i)|$ had to be $\rho(\ga_i)$-close to $|\partial\phi^*_i(x_i)|$ under the $\mF$-metric, therefore \[\lim|\partial \phi^1_i(x_i)|=\lim|\partial \phi^*_i(x_i)| \in A^c_\infty\cap C(S^*),\] and this finishes the proof of the claim.\\


Heuristically, one would like to set $\phi_i=\phi_i^1$ as the desired sequence, but since the isotopies $\Psi_{|\partial \phi^*_i(x)|}$ depend on $x$, the fineness of $\{\phi^1_i\}$ could be large even if $\mf(\phi^*_i)$ is small. Thus we need to interpolate $\phi^1_i$ to get the desired $\phi_i$, but we need to make sure the values of $\phi_i$ after interpolation are $\mF$-close to those of $\phi^1_i$. The idea is to extend $\phi^1_i$ to a \textit{piecewise} continuous (with respect to the $\mF$-metric) map on $I$ and then apply the discretization result in \cite[Theorem 5.1]{Zhou15b}. 

Similar difficulties appeared in the same way in \cite[\S 15]{MN14}. 
Instead, we use another interpolation method in Claim 2 as developed in \cite{Zhou-Zhu17}. For completeness, we provide detailed proof in Appendix \ref{A:proof of claim 2}.

\vspace{0.3cm}
\textit{Claim 2: there exist integers $l_i>k_i$ and maps $\phi_i: I(1, l_i)_0\to (\C(M), \{0\})$ for each $i$, such that $S=\{\phi_i\}$ is homotopic to $S^*$, and
\begin{itemize}
\item[(a)] $\phi_i^1=\phi_i\circ \n(l_i, k_i)$ on $I(1, k_i)_0$;
\item[(b)] $\mf(\phi_i)\to 0$, as $i\to \infty$;
\item[(c)]  $\Ac(\phi_i(x))-\max\{\Ac(\phi^1_i(y)): \al\in I(1, k_i)_1, x, y\in \al\}\to 0$, uniformly in $x\in I(1, l_i)_0$ as $i\rightarrow \infty$. 
\item[(d)] $\max\{\mF(\partial\phi_i(x), \partial \phi^1_i(y)): \al\in I(1, k_i)_1, x, y\in \al\} \to 0$, as $i\rightarrow \infty$.
\end{itemize}}

\vspace{0.3cm}
In particular, $S$ is a valid sequence in $\Pi$, and we now check that it satisfies the requirements of the proposition. First, property (c) and the fact that $S^*$ is a critical sequence directly imply that $S$ is also a critical sequence. It remains to show that every element in $C(S)$ must lie in $C(S^*)$ and have $c$-bounded first variation. Given $V\in C(S)$, one can find a subsequence (without relabeling) $\{\phi_i(\overline{x}_i): \overline{x}_i\in I(1, l_i)_0\}\subset \C(M)$, such that $V=\lim |\partial \phi_i(\overline{x}_i)|$ as varifolds, and
\[ \lim\Ac(\phi_i(\overline{x}_i))=\bL^h(\Pi). \]
We will need to first consider $\phi_i(x_i)=\phi_i^1(x_i)$, where $x_i$ is the nearest point to $\overline{x}_i$ in $I(1, k_i)_0$. By (c) and (d), we have $\lim\Ac(\phi^1_i(x_i))=\bL^h(\Pi)$ and also $\lim |\partial \phi_i(\overline{x}_i)|=\lim |\partial \phi^1_i(x_i)|$ as varifolds. Then by Claim 1, we conclude that $V\in A^c_\infty\cap C(S^*)$. This completes the proof. 
\end{proof}

\section{$h$-Almost minimizing}
\label{S:c-Almost minimizing}

In this section, we introduce the notion of $h$-almost minimizing varifolds, and show the existence of such a varifold from min-max construction. We also construct $h$-replacements for any $h$-almost minimizing varifold. Using the properties of these replacements, we show that all blowups of $h$-almost minimizing varifolds are regular. As an easy consequence, the tangent cones of such varifolds are always planar. 

\begin{definition}[$h$-almost minimizing varifolds]
\label{D:c-am-varifolds}
Let $\nu$ be the $\F$ or $\M$-norm, or the $\mF$-metric. For any given $\ep, \de>0$ and an open subset $U\subset M$, we define $\sA^h(U; \ep, \de; \nu)$ to be the set of all $\Om\in\C(M)$ such that if $\Om=\Om_0, \Om_1, \Om_2, \cdots, \Om_m\in\C(M)$ is a sequence with:
\begin{itemize}
\item[(i)] $\spt(\Om_i-\Om)\subset U$;
\item[(ii)] $\nu(\partial\Om_{i+1}, \partial\Om_i)\leq \de$;
\item[(iii)] $\Ac(\Om_i)\leq \Ac(\Om)+\de$, for $i=1, \cdots, m$,
\end{itemize}
then $\Ac(\Om_m)\geq \Ac(\Om)-\ep$.

We say that a varifold $V\in\V_n(M)$ is {\em $h$-almost minimizing in $U$} if there exist sequences $\ep_i \to 0$, $\de_i \to 0$, and $\Om_i\in \sA^h(U; \ep_i, \de_i; \F)$, such that $\mF(|\partial\Om_i|, V)\leq \ep_i$.
\end{definition}

A simple consequence of the definition is that $h$-almost minimizing implies bounded first variation:
\begin{lemma}
\label{L:c-am implies c-bd-first-variation}
Let $V\in\V_n(M)$ be $h$-almost minimizing in $U$, then $V$ has $c$-bounded first variation in $U$, where $c= \sup|h|$. 
\end{lemma}

Following similar arguments to \cite[4.10]{P81}, we can show that every sequence which has been pulled tight using Proposition \ref{P:tightening} has a critical limit which is $h$-almost minimizing on small annuli:

\begin{definition}
\label{D:am-annuli}
A varifold $V \in \V_n(M)$ is said to be \emph{$h$-almost minimizing in small annuli} if for each $p\in M$, there exists $r_{am}(p) >0$ such that $V$ is $h$-almost minimizing in $A_{s, r}(p)\cap M$ for all $0<s<r\leq r_{am}(p)$, where $A_{s, r}(p)=B_r(p)\backslash B_s(p)$.
\end{definition}

\begin{theorem}[Existence of $h$-almost minimizing varifold] 
\label{T: existence of almost minimizing varifold}
Let $\Pi \in \pi^{\sharp}_1\big(\C(M, \M), \{0\}\big)$, $c=\sup |h|$ and assume that $\bL^h(\Pi)>0$. There exists a nontrivial $V\in\V_n(M)$, such that
\begin{itemize}
\item[(i)] $V\in C(S)$ for some critical sequence $S$ of $\Pi$;
\item[(ii)] $V$ has $c$-bounded first variation;
\item[(iii)] $V$ is $h$-almost minimizing in small annuli. 
\end{itemize}
\end{theorem}

In fact, one may show that for any critical sequence $S$ of $\Pi$, there exists a varifold $V\in C(S)$, so that for any small annulus $A$, there exists a min-max sequence $\{\Omega_i\}$ such that $\Omega_i$ are eventually in $\sA^h(A ; \epsilon_i, \delta_i; \M)$ for $\epsilon_i,\delta_i\to 0$, and have the nontrivial varifold limit $|\partial \Omega_i| \rightarrow V$. The main idea is that if there is no such sequence, the procedures described in \cite[4.10]{P81} allows one to deform $S$ homotopically to a new sequence $\tilde{S}$ with $\bL^h(\tilde{S})<\bL^h(S)$, which contradicts the criticality of $S$.

We omit the proof of Theorem \ref{T: existence of almost minimizing varifold} except to point out the following equivalence result among several almost minimizing concepts using the three different topologies. In particular, it implies that we can work with the $\M$-norm at the expense of shrinking the open subset $U \subset M$. 

\begin{proposition}
\label{P:def-equiv}
Given $V\in \V_n(M)$, the following statements satisfy $(a)\Longrightarrow (b)\Longrightarrow (c)\Longrightarrow (d)$:
\begin{itemize}
\item[$(a)$]  $V$ is $h$-almost minimizing in $U$;
\item[$(b)$]  For any $\ep>0$, there exists $\de>0$ and $\Om \, \in \, \sA^h(U; \ep, \de; \mF)$ such that $\mF(V, |\partial \Om|)<\ep$;
\item[$(c)$]  For any $\ep>0$, there exists $\de>0$ and $\Om \in \sA^h(U; \ep, \de; \M)$ such that $\mF(V, |\partial\Om|)<\ep$;
\item[$(d)$] $V$ is $h$-almost minimizing in $W$ for any relatively open subset $W \subset \subset U$.
\end{itemize}
\end{proposition}

\begin{remark} 
The proof of Proposition \ref{P:def-equiv} for the area functional was originally due to Pitts \cite[Theorem 3.9]{P81}. In our context, we work with boundaries instead of general integral currents. Furthermore, in Definition \ref{D:c-am-varifolds}(iii), we use the $\Ac$ functional instead of the mass $\M$. The main point is to derive $(c)\Longrightarrow (d)$, which follows from a standard interpolation process; see Lemma \ref{L:interpolation1}.
\end{remark}

Now we formulate and solve a natural constrained minimization problem. 
\begin{lemma}[A constrained minimization problem]
\label{L:minimisation}
Given $\ep, \de>0$, $U\subset M$ and any $\Om \in \sA^h(U;\ep,\de;\F)$, fix a compact subset $K\subset U$. Let $\C_\Om$ be the set of all $\La\in\C(M)$ such that there exists a sequence $\Om=\Om_0, \Om_1, \cdots, \Om_m=\La$ in $\C(M)$ satisfying:
 \begin{itemize}
\item[(a)] $\spt(\Om_i-\Om)\subset K$;
\item[(b)] $\F(\partial\Om_i - \partial\Om_{i+1})\leq \de$;
\item[(c)] $\Ac(\Om_i)\leq \Ac(\Om)+\de$, for $i=1, \cdots, m$.
\end{itemize}

Then there exists $\Omega^* \in \C(M)$ such that:
\begin{itemize}
\item[(i)] $\Om^* \in \C_\Om$, and \[ \Ac(\Om^*)=\inf\{\Ac(\La):\ \La\in\C_\Om\},\]
\item[(ii)] $\Om^*$ is locally $\Ac$-minimizing in $\interior(K)$,
\item[(iii)] $\Om^*\in \sA^h(U;\ep,\de;\F)$.
\end{itemize}
\end{lemma}
\begin{proof}
Let us first describe the construction of $\Om^*$. Take any minimizing sequence $\{\La_j\}\subset\C_{\Om}$, i.e.
\[ \lim_{j \to \infty} \Ac(\La_j) = \inf\{\Ac(\La):\ \La\in\C_{\Om}\}.\]
Notice that $\spt(\La_j-\Om) \subset K$ and $\Ac(\La_j) \leq \Ac(\Om) + \de$ for all $j$. By standard compactness \cite[Theorem 6.3]{Si83}, after passing to a subsequence, $\partial\La_j$ converges weakly to some $\partial\Om^*$ with $\Om^* \in\C(M)$ and $\spt(\Om^*-\Om) \subset K$.
Since $\partial\La_j$ converges weakly to $\partial\Om^*$, we have that $\mH^n(\partial\Om^*)\leq \lim_{j\to\infty}\mH^n(\partial\La_j)$ and $ \int_{\Om^*}h = \lim_{j\to \infty}\int_{\La_j} h $. Therefore,
\begin{equation}
\label{E:Om^*}
\Ac(\Om^*) \leq  \inf\{\Ac(\La):\ \La\in\C_{\Om}\}.
\end{equation}
Now by taking the advantage of the discrete deformation sequence, properties (i-iii) follow in the same way as those in \cite[Lemma 5.7]{Zhou-Zhu17}, and we leave the details to readers.
\end{proof}

Given an $h$-almost minimizing varifold $V$, we can construct our $h$-replacements by applying Lemma \ref{L:minimisation} to the approximating sequence as in Definition \ref{D:c-am-varifolds} and taking the limit:

\begin{proposition}[Existence and properties of replacements]
\label{P:good-replacement-property}
Let $V\in\V_n(M)$ be $h$-almost minimizing in an open set $U \subset M$ and $K \subset U$ be a compact subset, then there exists $V^{*}\in \V_n(M)$, called \emph{an $h$-replacement of $V$ in $K$} such that, with $c=\sup|h|$, 
\begin{enumerate}
\item[(i)] $V\lc (M\backslash K) =V^{*}\lc (M\backslash K)$;
\item[(ii)] $-c \vol(K)\leq \|V\|(M)-\|V^{*}\|(M) \leq c \vol(K)$;
\item[(iii)] $V^{*}$ is $h$-almost minimizing in $U$;
\item[(iv)] $V^{*} =\lim_{i \to \infty} |\partial\Om^*_i|$ as varifolds for some $\Om^*_i\in\C(M)$ such that $\Om^*_i\in \sA^h(U; \ep_i, \de_i; \F)$ with $\ep_i, \de_i \to 0$; furthermore $\Om^*_i$ locally minimizes $\Ac$ in $\interior(K)$;
\item[(v)] if $V$ has $c$-bounded first variation in $M$, then so does $V^*$.
\end{enumerate}
\end{proposition}

\begin{lemma}[Regularity of $h$-replacement]
\label{L:reg-replacement}
Let $2 \leq n \leq 6$. Suppose that $h$ satisfies $(\dagger)$ or $(\ddagger)$. Under the same hypotheses as Proposition \ref{P:good-replacement-property}, if $\Sigma=\spt \|V^*\| \cap \interior(K)$, then $\Sigma$ is a smooth, almost embedded, stable $h$-hypersurface. Furthermore, $V^* \lc \interior(K) = \sum_{i=1}^L m_i \Sigma_i,$ where each component $\Sigma_i$ satisfies either: 
\begin{enumerate}
\item $\Sigma_i$ is not minimal, and $m_i=1$ so the density of $V^*$ is $1$ along $\mR(\Si_i)$ and $2$ along $\mS(\Si_i)$; or 
\item $\Sigma_i$ is minimal and smoothly embedded, but $m$ is some natural number. 
\end{enumerate}
Finally, if $p\in \Sigma_i$ and $m_i =1$ then $\Sigma$ is locally a boundary in a neighborhood of $p$. 
\end{lemma}
\begin{proof}
By the regularity for local minimizers of the $\Ac$ functional (Theorem \ref{T:regularity of Ac minimizers}), we know that each $\partial\Om^*_i$ is a smoothly embedded $h$-boundary in $\interior(K)$ by Proposition \ref{P:good-replacement-property}(iv). Moreover, $\partial\Om^*_i$ is stable in $\interior(K)$ in the sense of Definition \ref{D:stable c-hypersurface}. If this were not true, one can deform $\Om^*_i$ by ambient isotopies supported in $\interior(K)$ such that the $\Ac$ values are strictly decreasing; it is then easy to see this contradicts Lemma \ref{L:minimisation}(i). The lemma then follows from the compactness Theorem \ref{T:compactness}.
\end{proof}

By iterating Proposition \ref{P:good-replacement-property}, we see that blowups of an $h$-almost minimizing varifolds have the good replacement property of \cite{CD03, DT13}. This allows us to characterize certain blowups of the the $h$-min-max varifold. In particular the tangent cones are planar; see also Section \ref{S:Regularity for c-min-max varifold}. The detailed proofs follow similar to \cite[Lemma 5.10, Proposition 5.11]{Zhou-Zhu17}, and we omit them here.

Given $p\in M, r>0$, let $\bleta_{p,r}: \R^L\to\R^L$ be the dilation defined by $\bleta_{p, r}(x)=\frac{x-p}{r}$.
\begin{proposition}
\label{L:blowup is regular}
Let $2\leq n\leq 6$, and $V\in\V_n(M)$ be an $h$-almost minimizing varifold in $U$. Given a sequence $p_i\in U$ with $p_i\to p\in U$ and, a sequence $r_i>0$ with $r_i\to 0$, let $\overline{V}=\lim (\bleta_{p_i, r_i})_\# V$ be the varifold limit. Then $\overline{V}$ is an integer multiple of some complete embedded minimal hypersurface $\Si$ in $T_p M$, and moreover, $\Si$ is proper.
\end{proposition}

\begin{proposition}[Tangent cones are planes]
\label{P:tangent-cone}
Let $2 \leq n \leq 6$. Suppose $V \in \V_n(M)$ has $c$-bounded first variation in $M$ and is $h$-almost minimizing in small annuli. Then $V$ is integer rectifiable. Moreover, for any $C \in \VarTan(V,p)$ with $p \in \spt\|V\|$, 
\begin{equation}
\label{E:tangent cones are planes}
C= \Theta^n (\|V\|, p) |S| \text{ for some $n$-plane $S \subset T_p M$ where $\Theta^n(\|V\|, p)\in\N$}.
\end{equation}
\end{proposition}

\section{Regularity for $h$-min-max varifold}
\label{S:Regularity for c-min-max varifold}

In this section, we prove the regularity of our min-max varifolds. In particular we prove that every varifold which has $c$-bounded variation and is $h$-almost minimizing in small annuli is a smooth, closed, almost embedded, hypersurface of $h$-prescribed mean curvature whose non-minimal components have multiplicity one.

\begin{theorem}[Main regularity]
\label{T:main-regularity}
Let $2 \leq n \leq 6$, and $(M^{n+1}, g)$ be an $(n+1)$-dimensional smooth, closed Riemannian manifold. Further let $h:M\rightarrow \mathbb{R}$ be a smooth function satisfying $(\dagger)$ or $(\ddagger)$, and set $c=\sup|h|$.
 
If $V \in \V_n(M)$ is a varifold which
\begin{enumerate}
\item has $c$-bounded first variation in $M$, and 
\item is $h$-almost minimizing in small annuli,
\end{enumerate}
then $V$ is induced by $\Sigma$, where $\Sigma$ is a closed, almost embedded $h$-hypersurface (possibly disconnected). Moreover, each component $\Sigma^{(i)}$ of $\Sigma$ satisfies either:
\begin{itemize}
\item[(i)] $\Sigma^{(i)}$ is not minimal, the density of $V$ is exactly $1$ at the regular set $\mR(\Si^{(i)})$ and $2$ at the touching set $\mS(\Si^{(i)})$; or
\item[(ii)] $\Sigma^{(i)}$ is a smoothly embedded minimal hypersurface (in particular $h\equiv 0$ on $\Sigma^{(i)}$), and $V$ has integer density $m_i$ on $\Sigma^{(i)}$. 
\end{itemize}
\end{theorem}

\begin{proof}
The conclusion is purely local, so we only need to prove the regularity of $V$ near an arbitrary point $p\in \spt\|V\|$. Fix a $p\in\spt\|V\|$, then there exists $0< r_0 < r_{am}(p)$ such that for any $r<r_0$, the mean curvature $H$ of $\partial B_r(p)\cap M$ in $M$ is greater than $c$. Here $r_{am}(p)$ is as in Definition \ref{D:am-annuli}.

In particular, if $r<r_0$ and $W\in\V_n(M)$ has $c$-bounded first variation in $M$ and $W\neq 0$ in $B_r(p)$, then by the maximum principle (Proposition \ref{P:maximum principle})
\begin{equation}
\label{E:corollary-max-principle}
\emptyset \neq \spt \|W\| \cap \partial B_r(p) = \Clos\left( \spt \|W\| \setminus \Clos(B_r(p))\right) \cap \partial B_r(p).
\end{equation} 
Note that in the second equality we need a localized version of Proposition \ref{P:maximum principle} which holds true by the remark after \cite[Theorem 2]{White10}.

We will show that $V\lc B_{r_0}(p)$ is either an embedded minimal hypersurface (on which $h\equiv 0$) with integer multiplicity, or an almost embedded hypersurface of prescribed mean curvature $h$ with density equal to $2$ along its touching set. 

The argument consists of five steps:

\vspace{.3cm}
\noindent \textbf{Step 1:} \textit{Constructing successive $h$-replacements $V^*$ and $V^{**}$ on two overlapping concentric annuli.}

\noindent \textbf{Step 2:} \textit{Gluing the $h$-replacements smoothly (as immersed hypersurfaces) on the overlap.}

\noindent \textbf{Step 3:} \textit{Extending the $h$-replacements to the point $p$ to get a $h$-`replacement' $\ti{V}$ on the punctured ball.}

\noindent \textbf{Step 4:} \textit{Showing that the singularity of $\ti{V}$ at $p$ is removable, so that $\ti{V}$ is regular.}

\noindent \textbf{Step 5:} \textit{$V$ coincides with the almost embedded hypersurface $\ti{V}$ on a small neighborhood of $p$.}
\vspace{.3cm}

We now proceed to the proof.

\subsection*{Step 1}
We first describe the construction of $h$-replacements on overlapping annuli; a key property will be that the replacements are also boundaries in the chosen annulus (see Claim 1), at least near points of multiplicity one.  


Fix $0<s<t<r_0$. By the choice of $r_0$, we can apply Proposition \ref{P:good-replacement-property} to $V$ to obtain an $h$-replacement $V^*$ in $K=\Clos(A_{s, t}(p)\cap M)$.  By (\ref{E:corollary-max-principle}) and Lemma \ref{L:reg-replacement}, the restriction
\[ \Si_1=\spt \|V^*\| \lc (A_{s, t}(p)\cap M) \]
is a nontrivial, smooth, almost embedded, stable $h$-hypersurface with some unit normal $\nu_1$; when the multiplicity is $1$, $\Si_1$ is locally a boundary so we can choose $\nu_1$ to be the outer normal. 

By Proposition \ref{P:smooth touching set}, the touching set $\mathcal{S}(\Sigma_1)$ is contained in a countable union of $(n-1)$-dimensional connected submanifolds $\bigcup S_1^{(k)}$. Since a countable union of sets of measure zero still has measure zero, it follows from Sard's theorem that we may choose $s_2 \in (s, t)$ such that $\partial (B_{s_2}(p)\cap M)$ intersects $\Si_1$ and all the $S_1^{(k)}$ transversally. 

Given any $s_1\in (0, s)$, by Proposition \ref{P:good-replacement-property}(iii), $V^*$ is still $h$-almost-minimizing in small annuli, and we can apply Proposition \ref{P:good-replacement-property} again to get an $h$-replacement $V^{**}$ of $V^*$ in $K=\Clos(A_{s_1, s_2}(p)\cap M)$. By (\ref{E:corollary-max-principle}) and Lemma \ref{L:reg-replacement} again, the restriction
\[ \Si_2=\spt\|V^{**}\|\lc (A_{s_1, s_2}(p)\cap M) \]
is also a nontrivial, smooth, almost embedded, stable $h$-hypersurface with some unit normal $\nu_2$, which points outward at multiplicity $1$ points in $A_{s_1,s_2}(p)$. Note that by Proposition \ref{P:good-replacement-property}(v), both $V^*$ and $V^{**}$ have $c$-bounded first variation.

We first observe that the second $h$-replacement can be chosen to be locally given by a boundary near `multiplicity one' points in $\partial B_{s_2}(p)$, and in particular near non-minimal components of $\Sigma_2$. 

\vspace{.3cm}
\textit{Claim 1: there exists a set $\Om^{**}\in\C(M)$ satisfying the following: Suppose $q\in \spt(V^{**})\cap \partial B_{s_2}(p)$, and $\Sigma_1$, $\Sigma_2$ have multiplicity $1$ in a neighborhood of $q$.
Then there exists $\epsilon>0$ so that 
\begin{itemize}
\item[a)] $\Si_1$ and $\Si_2$ are the boundaries of $\Om^{**}$ in $A_{s_2, t}(p)\cap B_\epsilon(q)$ and $A_{s_1, s_2}(p)\cap B_\epsilon(q)$ respectively;
\item[b)] $\nu_1, \nu_2$ coincide with the outer unit normal of $\Om^{**}$ in $A_{s_2, t}(p)\cap B_\epsilon(q)$ and $A_{s_1, s_2}(p)\cap B_\epsilon(q)$ respectively;
\item[c)] if $\|V^{**}\|(\partial B_{s_2}(p))=0$, then $V^{**}$ is identical to $|\partial\Om^{**}|$ in $B_\epsilon(q)\cap M$.
\end{itemize}}

The claim follows by constructing the $\epsilon_i$-replacements $\Omega_i^{**}\rightarrow V^{**}$ from corresponding replacements $\Omega_i^*\rightarrow V^*$ (in particular, with consistent orientations).  We refer the reader to \cite[Theorem 6.1, Claim 1]{Zhou-Zhu17} for the details.

\subsection*{Step 2}

We now show that $\Sigma_1$ and $\Sigma_2$ glue smoothly (as immersed hypersurfaces) across $\partial (B_{s_2}(p)\cap M)$. 
Indeed, define the intersection set
\begin{equation}
\Ga=\Si_1\cap \partial(B_{s_2}(p)\cap M),\quad \mS(\Ga)=\Ga\cap \mS(\Si_1).
\end{equation}
Then by transversality, $\Ga$ is an almost embedded hypersurface in $\partial(B_{s_2}(p)\cap M)$, and $\mS(\Ga)$ is its touching set.  Notice that
\begin{equation}
\label{E:openness of Gamma}
\text{$\mS(\Ga)$ is closed, and $\mR(\Ga)=\Ga\backslash\mS(\Ga)$ is open in $\Ga$.}
\end{equation}
It follows from the maximum principle that 
\[ \Clos(\Si_2)\cap \partial(B_{s_2}(p)\cap M)\subset\Ga. \] 
Indeed, (\ref{E:corollary-max-principle}) implies that any $y \in \Clos(\Si_2)\cap \partial(B_{s_2}(p)\cap M)$ is also a limit point of $\spt\|V^{**}\|$ from the outer side of $\partial B_{s_2}(p)$, on which $V^{**}$ coincides with $\Sigma_1$. In fact, with a little more work we have

\vspace{.3cm}
\textit{Claim 2: $\Clos(\Si_2)\cap \partial(B_{s_2}(p)\cap M)=\Ga$, and then $\Si_1$ glues together continuously with $\Si_2$.}
\vspace{.3cm}

\textit{Proof of Claim 2:}  By Proposition \ref{P:good-replacement-property}(i), we have
\begin{equation}
\label{E:agreement-with-sigma1}
V^*=V^{**}=\text{$\Si_1$,}\quad \text{ in $A_{s_2, t}(p)\cap M$.}
\end{equation} 
Given any $x\in \Ga$, using (\ref{E:agreement-with-sigma1}), Proposition \ref{P:tangent-cone} and the fact that $\Si_1$ meets $\partial (B_{s_2}(p)\cap M)$ transversally, we have
\begin{equation}
\label{E:tangent cones for the second replacement}
\VarTan(V^{**}, x) = \{ \Theta^n(\|V^*\|, x) \; |T_x \Si_1| \}.
\end{equation}
This implies that $x$ is a limit point of $\spt\|V^{**}\|$ from inside of $\partial B_{s_2}(p)$, and thus completes the proof of the claim. \qed
\vspace{.3cm}

\vspace{.3cm}
Furthermore, we will show that $\Si_1$ glues with $\Si_2$ in $C^1$, i.e. the tangent spaces of $\Si_1$ and $\Si_2$ agree along $\Ga$, with matching normals. 

First we have the following.

\vspace{.3cm}
\textit{Claim 3: Fix $x\in\Gamma$, and denote the plane $P_x = T_x\Sigma_1$ (with multiplicity 1). Then for any sequence of $x_i\to x$ with $x_i\in \Ga$ and $r_i\to 0$, up to a subsequence we have}
\[ \lim_{i \to \infty} (\bleta_{x_i,r_i})_\sharp V^{**} = \begin{cases}
\Theta^n(\|V^*\|,x) P_x, &  \text{if } x\in \mR(\Gamma), \\
P_x+\btau_v P_x, & \text{if } x\in \mS(\Gamma), \liminf_{i \to \infty} \dist_{\R^L}(x_i, \mS(\Ga))/r_i=\infty,\\
2P_x, &  \text{if } x\in \mS(\Gamma), \liminf_{i \to \infty} \dist_{\R^L}(x_i, \mS(\Ga))/r_i<\infty,\end{cases} \]
\textit{where $\btau_w$ denotes translation by a vector $w$, and $v\in (P_x)^\perp$ is a vector in $T_xM$ orthogonal to $P_x$ ($v$ may be $\infty$, in which case $\btau_v P$ is understood to be the empty set).}

\vspace{.3cm}

The proof of Claim 3 follows from the half space theorem for minimal hypersurfaces \cite[Theorem 3]{HM90} and the classical maximum principle for minimal hypersurfaces, after using that $V^*$ and $V^{**}$ coincide on $A_{s_2,t}(p)$ to determine the blowup on a halfspace. We refer the reader to \cite[Theorem 6.1, Claim 3(A)(B)]{Zhou-Zhu17} for the details.

\vspace{.3cm}

Since $\{(\bleta_{x_i,r_i})_\sharp V^{**}: i\in\N\}$ have uniformly bounded first variation, a standard argument using the monotonicity formula implies that, in the Hausdorff topology, either 
\begin{equation}
\label{E:Hausdorff-convergence}
\spt \|(\bleta_{x_i,r_i})_\sharp V^{**}\| \to \begin{cases} T_x \Sigma_1, \text{or}&\\ T_x\Sigma_1 + \btau_v T_x\Sigma_1, & v \in (T_x\Sigma_1)^\perp.\end{cases}
\end{equation}

To show that $\Si_1$ and $\Si_2$ glue together along $\Ga$ in $C^1$ near $q$, we will need:

\vspace{.3cm}

\textit{Claim 4: For each $x\in\Gamma$, we have}
\begin{equation}
\label{E:convergence of tangent planes}
\lim_{z \to x, z \in \Si_2} [T_z\Si_2] = [T_x\Si_1],
\end{equation}
\textit{where $[T_z\Si_2]$ and $[T_x\Si_1]$ denote the un-oriented tangent planes of $\Si_2$ and $\Si_1$ (without counting multiplicity) respectively; and the convergence is uniform in $x$ on compact subsets of $\Gamma$.}

\textit{Moreover, if $x\in \mR(\Gamma)$ and $x$ lies in a multiplicity 1 component of $\Sigma_1$, 
then in fact $\nu_2(z)\rightarrow \nu_1(x)$. }

\vspace{.3cm}

\textit{Proof of Claim 4:} The uniformity follows from the fact that $[T_\cdot\Sigma_1]$ is continuous on $\Gamma$, so we only need to establish the convergence to $[T_\cdot\Sigma_1]$.  

So consider a sequence $z_i\in \Sigma_2$ converging to some $x\in \Ga$. 
Take $x_i \in \Ga$ to be the nearest point projection (in $\R^L$) of $z_i$ to $\Ga$ and $r_i=|z_i-x_i|$. Note that $x_i \to x \in \Ga$ and $r_i \to 0$, so we are in the situation of Claim 3. Note that $\Si_2 \cap B_{r_i/2}(z_i)$ is an almost embedded, stable $h$-hypersurface in $M$, so by Theorem \ref{T:compactness} a subsequence of the blow-ups $\bleta_{x_i,r_i} (\Si_2 \cap B_{r_i/2}(z_i))$ converges smoothly to a smooth, embedded, stable, minimal hypersurface $\Si_\infty$ contained in a half-space of $T_x M$. 

On the other hand, Claim 3 and (\ref{E:Hausdorff-convergence}) imply that $\bleta_{x_i,r_i} (\Si_2 \cap B_{r_i/2}(z_i))$ converges in the Hausdorff topology to a domain in $T_x \Si_1$. Therefore, we have $\Si_\infty \subset T_x \Si_1$. The smooth convergence then implies that $\nu_2(z_i)$ converges to one of the unit normals $\pm\nu_1(x)$ of $T_x \Sigma_1$. 

Now suppose that $x\in \mR(\Gamma)$ and $x$ lies in a multiplicity 1 component of $\Sigma_1$. Then the multiplicity of $\Sigma_2$ is also $1$ near $x$, by the upper semi-continuity of density function for varifolds with bounded first variation \cite[17.8]{Si83}. Then by (\ref{E:tangent cones for the second replacement}), \cite[Theorem 3.2(2)]{Si83} and Claim 1(c), we have
\begin{equation}
\label{E:V** is boundary}
\text{$\|V^{**}\|(\partial B_{s_2}(p))=0$, and hence $V^{**}=|\partial\Om^{**}|$ in $B_\epsilon(x) \cap M$ for some $\epsilon>0$.}
\end{equation}
This implies that the limit of $\nu_2(z_i)$ must be $\nu_1(x)$ by Claim 1, so \emph{Claim 4} is proved. \qed

\vspace{.3cm}

We first consider the gluing near a regular point $q\in \mR(\Ga)$. By Claim 4, $\Sigma_1$ and $\Sigma_2$ glue together along $\Gamma$ as a $C^1$ hypersurface with matching unit normals and prescribed mean curvature $h$ (in the weak sense). Note that if either $\Sigma_i$ was a minimal hypersurface near $q$ then we may have had to choose the opposite normal to $\nu_i$ to ensure the $C^1$ gluing. (This is allowed since in this situation $h$ vanishes on $\Sigma_i$, so it still has prescribed mean curvature $h$ with respect to the opposite normal.)

The higher regularity then follows from a standard elliptic PDE argument. More precisely, $\Si_1$ and $\Si_2$ can be written as graphs of some functions $u_1, u_2$ over $T_q\Si_1$ respectively. Since $\Sigma_1$ and $\Sigma_2$ both have prescribed mean curvature $h$ with respect to unit normals pointing to the same side of $T_q\Si_1$, they satisfy the same mean curvature type elliptic PDE with inhomogeneous term given by $h|_{\Sigma_1\cup \Sigma_2}$, which has the same regularity as the glued hypersurface $\Sigma_1\cup \Sigma_2$. The higher regularity follows from the elliptic regularity of this PDE. 


Thus we have proven that $\Si_2$ glues smoothly with $\Si_1\cap A_{s_2, t}(p)$ along $\mR(\Ga)$. Moreover, by the unique continuation for elliptic PDE (for instance Corollary \ref{cor:uniqcont}) we know that $\Si_2$ is identical to $\Si_1$  in a neighborhood of $\mR(\Ga)$ in $A_{s_2, t}(p)\cap M$. We now show that the smooth gluing extends to the touching set $\mS(\Ga)$.

\vspace{.3cm}



Now consider an arbitrary fixed singular point $q\in \mS(\Gamma)$. By Lemma \ref{L:reg-replacement}, in some small neighborhood $U\subset M$ of $q$, $\Si_1\cap U$ is the union of two connected, embedded $h$-hypersurfaces $\Si_{1,1}\cup \Si_{1,2}$ with unit normals $\nu_{1,1}$ and $\nu_{1,2}$, such that $\Si_{1,2}$ lies on one side of $\Si_{1,1}$ and they touch tangentially at $\mS(\Si_1)\cap U=\Si_{1,1}\cap\Si_{1,2}$. By Lemma \ref{lem:sameorientation}, 
$\nu_{1,1}=-\nu_{1,2}$ along the touching set $\mS(\Si_1)\cap U$. Denote $\Ga\cap \Si_{1,1}=\Ga_1$ and $\Ga\cap \Si_{1,2}=\Ga_2$, then as embedded submanifolds of $\partial (B_{s_2}(p)\cap U)$, $\Ga_2$ lies on one-side of $\Ga_1$ and they touch tangentially along $\mS(\Ga)\cap U$.

By Claim 4, using the regularity of $\Si_2$ in Lemma \ref{L:reg-replacement}, near $q$ the hypersurface $\Si_2$ can be written as a set of graphs $\{\Si_{2, i}: i=1\cdots l\}$ over the half space $[T_q(\Si_1\cap B_{s_2}(p))]$. 
Now since $\Si_2$ glues smoothly with $\Si_1$ along $\mR(\Ga)$, and since $\mR(\Ga)$ is an open and dense subset of $\Ga$, we know that the set $\{\Si_{2, i}: i=1\cdots l\}$ consists of exactly two elements: one of them, denoted by $\Si_{2,1}$, glues smoothly with $\Si_{1,1}$ along $\Ga_1\backslash \mS(\Ga)$; the other one, denoted by $\Si_{2, 2}$, glues smoothly with $\Si_{1,2}$ along $\Ga_2\backslash \mS(\Ga)$. This, together with (\ref{E:convergence of tangent planes}), implies that the pairs $(\Si_{1,1}, \Si_{2,1})$ and $(\Si_{1,2}, \Si_{2,2})$ glue together in $C^1$ near $q$ respectively (with matching orientations). Again higher regularity follows from the elliptic PDE argument as for the regular part. This completes the smooth gluing near the touching set.

\subsection*{Step 3}
We now wish to extend the replacements, via unique continuation, all the way to the center $p$. 

Henceforth we denote $V^{**}$ by $V^{**}_{s_1}$ and $\Si_2$ by $\Si_{s_1}$ to indicate the dependence on $s_1$. By varying $s_1\in (0, s)$, we obtain a family of nontrivial, smooth, almost embedded, stable $h$-hypersurfaces $\{\Si_{s_1}\subset A_{s_1, s_2}(p)\cap M\}$. Since unique continuation holds for immersed prescribed mean curvature hypersurfaces (Corollary \ref{cor:uniqcont}) by Step 2 we have $\Si_{s_1}=\Si_1$ in $A_{s, s_2}(p)$, and moreover, for any $s_1'<s_1<s$, we have $\Si_{s_1'}=\Si_{s_1}$ in $A_{s_1, s_2}(p)$. Hence 
\[ \Si:=\bigcup_{0<s_1<s} \Si_{s_1}   \]
is a nontrivial, smooth, almost embedded, stable $h$-hypersurface in $(B_{s_2}(p)\backslash\{p\})\cap M$. 

Let $\Sigma^{(i)}$, $i=1,\cdots, l$ be the connected components of $\Sigma$.  Then by (\ref{E:tangent cones for the second replacement}) we also have 
\[\spt\|V^{**}_{s_1}\|=\Si, \text{ in } A_{s_1, s_2}(p), V^{**}_{s_1}=V^* \text{ in } A_{s, s_2}(p),\]
\[ \text{and for any } s_1'<s_1<s, V^{**}_{s_1'}=V^{**}_{s_1} \text{ in } A_{s_1, s_2}(p). \]
By Proposition \ref{P:good-replacement-property}, $V^{**}_{s_1}$ has $c$-bounded first variation and uniformly bounded mass for all $0<s_1<s$. 
Therefore as $s_1\to 0$, the family $V^{**}_{s_1}$ will converge to a varifold $\ti{V}\in\V_n(M)$ with $c$-bounded first variation, i.e.
\[\ti{V}=\lim_{s_1\to 0} V^{**}_{s_1},\quad \text{such that}\]
\begin{equation}
\label{E:tilde V}
\ti{V} =\left\{ \begin{array}{cl}
\sum_{i=1}^l m_i \Sigma^{(i)} & \text{in } (B_{s_2}(p)\backslash\{p\})\cap M \\
V^* & \text{in } M\backslash B_s(p) \end{array}, \text{ and } \|\ti{V}\|(\{p\})=0,\right. 
\end{equation}
where $m_i>1$ only if $\Sigma^{(i)}$ is minimal.  Since $p\in\spt\|V^{**}_{s_1}\|$, by the upper semi-continuity of density function for varifolds with bounded first variation \cite[17.8]{Si83}, we know that $p\in\spt\|\ti{V}\|$.

\subsection*{Step 4}

We now determine the regularity of $\ti{V}$ at $p$.

First, $\ti{V}$ has $c$-bounded first variation. Second, $\spt \|\ti{V}\|$, when restricted to any small annulus $A_{r, 2r}(p)\cap M$, already coincides with a smooth, almost embedded, stable $h$-hypersurface $\Si$.  Using these two ingredients, we can use a standard blowup argument (see for instance Proposition \ref{P:tangent-cone} or \cite[Proposition 5.11]{Zhou-Zhu17}) to show that every tangent varifold of $\ti{V}$ at $p$ is an integer multiple of some $n$-plane, i.e. for any $C \in \VarTan(V,p)$, 
\[ C= \Theta^n (\|\ti{V}\|, p) |S|, \text{ for some $n$-plane $S \subset T_p M$ where $\Theta^n(\|\ti{V}\|, p)\in\N$}. \]

Now the removability of the singularity of $\ti{V}$ at $p$ (as an almost embedded hypersurface) follows similarly to \cite[Theorem 7.12]{P81}. We include the details for completeness. We can assume that
\[ \Theta^n (\|\ti{V}\|,p)=m \]
for some $m \in \mathbb{N}$. Since $\Sigma$ is stable in a punctured ball of $p$, by Theorem \ref{T:compactness}, for any sequence $r_j \to 0$, 
\[ \bleta_{p,r_j}(\Sigma) \to m\cdot S \]
locally smoothly in $\mathbb{R}^L \setminus \{0\}$ for some $n$-plane $S \subset T_p M$. However, $S$ may depend on the sequence $r_j$. By the convergence and the regularity of $\Si$, there exists $\si_0>0$ small enough, such that for any $0<\si\leq \si_0$, $\Si$ has an ordered (in the sense of Definition \ref{def:comparisonsheets}) graphical decomposition in $A_{\si/2, \si}(p)$:
\begin{equation}
\label{E: annuli-decomposition}
\ti{V}\lc A_{\si/2, \si}(p)=\sum_{i=1}^{l}m_i|\Si_i(\si)|, \qquad \sum_{i=1}^l m_i = m 
\end{equation}
Here $\Si_i(\si)$ are distinct graphs over $A_{\si/2, \si}(p)\cap S$ for some $n$-plane $S\subset T_p M$, and $m_i>1$ only if $\Sigma_i(\sigma)$ is minimal.

Since (\ref{E: annuli-decomposition}) holds for all $\si$, by continuity of $\Si$ we can continue each $\Si_i(\si_0)$ to $(B_{\si_0}(p)\setminus\{p\})\cap M$, and we denote this continuation by $\Si_i$. Since each piece $\Si_i$ has prescribed mean curvature $h$, by a standard extension argument (c.f. the proof in \cite[Theorem 4.1]{HL75}), each $\Si_i$ can be extended as a varifold with $c$-bounded first variation in $B_{\si_0}(p)\cap M$ (recall $c=\sup|h|$). 
Given $C_i\in \VarTan(\Si_i, p)$, to see that $C_i$ has multiplicity one, first notice that
\begin{equation}
\label{E:density=1}
\Theta^n(\|C_i\|,p) \geq 1, 
\end{equation}
since each $\Si_i$ is $h$-stable and thus its re-scalings converge with multiplicity to a smooth, embedded, stable, minimal hypersurface by Theorem \ref{T:compactness}. If equality does not hold for some $i$ in (\ref{E:density=1}), this will derive a contradiction since 
\[ \tilde{V}\lc B_{\si_0}(p)= \sum_{i=1}^l m_i |\Si_i|, , \qquad \sum_{i=1}^l m_i = m \]
Therefore, each $\Si_i$ has $c$-bounded first variation in $B_{\si_0}(p) \cap M$ and $\Theta^n(\||\Sigma_i|\|,p)=1$; by the Allard regularity theorem  \cite[Theorem 24.2]{Si83} and elliptic regularity, $\Si_i$ extends as a smooth, embedded $h$-hypersurface across $p$. Moreover, each minimal $\Sigma_i$ extends as a smooth minimal hypersurface across $p$. Finally, by the maximum principle (Lemma \ref{lem:sameorientation}, Lemma \ref{lem:minsheet}) 
either some sheet $\Sigma_i$ is indeed minimal so all sheets must coincide and $\tilde{V} \lc B_{\sigma_0}(p) = m [\Sigma_i]$, or none are minimal and $m=1$ or $m=2$. This shows that $\ti{V}$ extends as an almost embedded $h$-hypersurface across $p$ with the desired regularity.

\subsection*{Step 5}

Finally, to complete the proof we show that $V$ coincides with $\ti{V}$ on a small ball about $p$. 

We will need the following simple corollary of the first variation formula.

\begin{lemma} 
\label{L: tangent set of V}
For small enough $r$ the set
\[ \text{Tr}^V_p=\left\{ y \in \spt\|V\|\cap (B_r(p)\backslash\{p\}) : \begin{array}{l}
\VarTan(V, y) \text{ consists of an integer multiple of an }\\
\text{$n$-plane transverse to } \partial (B_{\dist(y, p)}(p)\cap M)
\end{array} \right\}\]
is a dense subset of $\spt \|V\|\cap B_r(p)$.
\end{lemma}

\textit{Claim 5: For small enough $r$, $\spt\|V\|=\Si$ in the punctured ball $(B_r(p)\backslash\{p\})\cap M$.}
 
\vspace{.3cm}
\textit{Proof of Claim 5:}  We first prove that $\text{Tr}^V_p\subset \Si$, which combined with Lemma \ref{L: tangent set of V} will imply that $\spt\|V\|\cap (B_r(p)\backslash\{p\})\subset \Si.$ Fix $y \in \text{Tr}^V_p \cap (B_r(p) \backslash \{p\})$, and let $\rho=\dist(y, p)$. Consider $V^{**}_\rho$. By transversality we have $ y  \in \Clos(\spt\|V\|\cap B_\rho(p))$. On the other hand, $V^{**}_{\rho}=V^*=V$ inside $B_\rho(p)$, so by (\ref{E:corollary-max-principle}) we have
\begin{eqnarray}\nonumber \Clos(\spt\|V\|\cap B_\rho(p)) \cap \partial B_\rho(p)&=&\Clos(\spt\|V^{**}_\rho\|\cap B_\rho(p)) \cap \partial B_\rho(p)\\&\subset&\nonumber \Clos\left( \spt \|V^{**}_\rho\| \setminus \Clos(B_\rho(p))\right) \cap \partial B_\rho(p).\end{eqnarray}
Since $\spt \|V^{**}_{\rho}\|= \Si$ on $A_{\rho, s_2}(p)$, we therefore have $y \in \Si$. 

Next we show the reverse inclusion $\Si\subset \spt\|V\|$. Since $\Si$ extends across $p$ as an almost embedded hypersurface, we know that $T_y\Si$ is transverse to $\partial (B_{\dist(y, p)}(p)\cap M)$ for all $y\in\Si\cap B_r(p)$ for small enough $r$. 
Let $\rho$ and $V^{**}_\rho$ be as above, then $y\in \spt\|V^{**}_\rho\|$. By Proposition \ref{P:tangent-cone}, $\VarTan(V^{**}_\rho, y)=\{ \Theta^n(\|V^{**}_\rho\|, y) |T_y \Si| \}$. By the transversality, we then have $y\in \Clos(\spt\|V^{**}_\rho\|\cap B_\rho(p))$, so since $V^{**}_\rho=V$ inside $B_\rho(p)$ we conclude that $y\in\Clos(\spt\|V\|\cap B_\rho(p))\subset \spt\|V\|$ as desired. \qed

\vspace{.3cm}

Note that we do not have a suitable Constancy Theorem (c.f. \cite[41.1]{Si83}) for varifolds with bounded first variation. In order to show that $V$ coincides with $\Si$ near $p$, our strategy is to show that $V=\ti{V}$ as varifolds in a neighborhood of $p$. By the transversality argument as above, we only need to show that the densities of $V$ and $\ti{V}$ are identical along $\Si\cap (B_r(p)\backslash\{p\}$, where $r$ is chosen as in Claim 5. 

\vspace{.3cm}
\textit{Claim 6: $\Theta^n(\|V\|, \cdot)=\Theta^n(\|\ti{V}\|, \cdot)$ on $\Si\cap B_r(p)\backslash\{p\}$.} 

\vspace{.3cm}
\textit{Proof of Claim 6:} 
Let $y\in\Si$ and $\rho=\dist(y,p) <r$ be as above. Then since $V^{**}_\rho=V$ inside $B_\rho(p)$, by transversality and Proposition \ref{P:tangent-cone} we have $\VarTan(V, y)=\VarTan(V^{**}_\rho, y)$. But $V^{**}_\rho = \tilde{V}$ on $A_{\rho,s_2}(p)$, so we must have $\VarTan(V^{**}_\rho, y)=\{ \Theta^n(\|\ti{V}\|, y)|T_y \Si| \}$. Thus $\Theta^n(\|V\|, y)=\Theta^n(\|\ti{V}\|, y)$.  \qed

Combining Claims 5 and 6 yields that $V=\tilde{V}$ on $B_r(p)\cap M$. This finishes the proof of Step 5, and hence also completes the proof of the main Theorem \ref{T:main-regularity}.
\end{proof} 

Finally we have the following result which is implied by the above proof. 
\begin{proposition}
Let $V$ be as in Theorem \ref{T:main-regularity}. Assume that $V=\lim_{i\to\infty}|\partial\Om_i|$, where $\{\Om_i\}$ are the approximating sets of the $h$-almost-minimizing varifold $V$, given by Definition \ref{D:c-am-varifolds}. If there are no minimal sheets in $\Sigma=\spt V$, then $\Sigma$ is a boundary of some $\Om\in\C(M)$ where its mean curvature with respect to the unit outer normal is $h$, and
\[ \Ac(\Om)=\lim_{i\to\infty}\Ac(\Om_i). \]
In particular, if $V\in C(S)$ for some critical sequence $S\in\Pi$ as given by Theorem \ref{T: existence of almost minimizing varifold} and the remarks following it, we have
\[  \Ac(\Om)=\bL^h(\Pi).  \]
\end{proposition}
\begin{proof}
It suffices to prove that $\partial\Om_i$ sub-converges to $\Sigma$ weakly as currents. Take $\Om$ as the weak limit of $\Om_i$ as Caccioppoli sets (up to a subsequence), then $\spt(\partial\Om)\subset \Sigma$, and $|\partial\Om|=0$ or $\Sigma$ as varifolds by the Constancy Theorem (by taking $\Sigma$ as the ambient space). Fix an arbitrary regular (non-touching) point $p$ in $\Sigma$. As in Step 1 of the proof of Theorem \ref{T:main-regularity}, we can take a first replacement $V^*$ near $p$. In fact, we showed that \textit{a posteriori} $V^*$ coincides with $V$. 

The construction of $V^*$, however, came with the constrained minimizers $\Om_i^*$ by Proposition \ref{P:good-replacement-property}. The Constancy Theorem still implies that $\Om_i^*$ converges weakly to $\Om$, but now according to the proof of Lemma \ref{L:reg-replacement}, the $\partial\Om_i^*$ are all smoothly embedded $h$-hypersurfaces, converging smoothly to $\Sigma$ in $\interior(K)$. Therefore $\partial\Om_i^*$ must converge to $\Sigma$ as currents in $\interior(K)$. This shows that $\partial\Om=\Sigma$ inside $\interior(K)$ as currents, and hence they must coincide everywhere, which concludes the proof.
\end{proof}


\appendix

\section{An interpolation lemma}
\label{A:An interpolation lemma}

The following interpolation lemma was proved in in \cite[Appendix A]{Zhou-Zhu17} when $h$ is a constant function. This type of result was essentially due to Pitts \cite[Lemma 3.8]{P81}, but the modification to find the interpolation sequence using boundaries of Caccioppoli sets was completed by the first author \cite[Proposition 5.3]{Zhou15b}. The extension to non-constant functions $h$ proceeds similar to \cite[Appendix A]{Zhou-Zhu17}, the key being to control $\left|\int_\Omega h\right|\leq c\vol(\Omega)$, so we omit the details here.
\begin{lemma}
\label{L:interpolation1}
Suppose $L>0$, $\eta>0$, $W$ is a compact subset of $U$, and $\Om\in\C(M)$. Then there exists $\de=\de(L, \eta, U, W, \Om)>0$, such that for any $\Om_1, \Om_2\in \C(M)$ satisfying
\begin{itemize}
\item[$(a)$] $\spt(\Om_i-\Om)\subset W$, $i=1, 2$,
\item[$(b)$] $\M(\partial \Om_i)\leq L$, $i=1, 2$,
\item[$(c)$] $\F(\partial\Om_1-\partial\Om_2)\leq \de$,
\end{itemize}
there exist a sequence $\Om_1=\La_0, \La_1, \cdots, \La_m=\Om_2\in \C(M)$ such that for each $j=0, \cdots, m-1$,
\begin{enumerate}[label=(\roman*)]
\item $\spt(\La_j-\Om)\subset U$;
\item $\Ac(\La_j)\leq \max\{\Ac(\Om_1), \Ac(\Om_2)\}+\eta$;
\item $\M(\partial\La_j-\partial\La_{j+1})\leq \eta$.
\end{enumerate}
\end{lemma}

\section{Interpolation process}
\label{A:proof of claim 2}

\begin{proof}[Proof of Claim 2 in Proposition \ref{P:tightening}]
Here we describe the construction of $\{\phi_i\}$ by interpolating $\{\phi^1_i\}$. 

Fix $i\in\N$ and consider a $1$-cell $\al\in I(1, k_i)$. We only need to show how to interpolate $\phi^1_i$ when restricted to $\al_0$. For notational simplicity we write $\al=[0, 1]$. For $x\in \alpha$ let $\ti{X}(x)$ be the linear interpolation between $\ti{X}_i(0)=\ti{X}(|\partial\phi^*_i(0)|)$ and $\ti{X}_i(1)=\ti{X}(|\partial\phi^*_i(1)|)$. The continuity of the map $V\to \ti{X}(V)$ implies that $\|\ti{X}_i(x)-\ti{X}_i(0)\|_{C^1(M)}\to 0$ uniformly as $i\to\infty$. Define $\bar{Q}_i(x)$ to be the push-forward of $\phi^*_i(0)$ by the flow of $\ti{X}_i(x)$ up to time $1$; this gives a map $\bar{Q}_i: \alpha \rightarrow \mathcal{C}(M)$. Note that $\partial\bar{Q}_i: \al\to \Z_n(M)$ is continuous under the $\mF$-metric.

Since $\bar{Q}_i(x)$ and $\phi^1_i(0)$ are the push-forwards of the same initial set $\phi^*_i(0)$ under the flows of $\ti{X}_i(x)$ and $\ti{X}_i(0)$ respectively, we have
\begin{equation}
\label{E:mF fineness control1}
\left. \begin{array}{cl}
\mF(\partial \bar{Q}_i(x), \partial \phi^1_i(0))\to 0 \\
\M(\bar{Q}_i(x)- \phi^1_i(0))\to 0
\end{array} \right.,  \text{ uniformly in $x,\alpha$ as $i\to\infty$}.
\end{equation}
As $\bar{Q}_i(1)$ and $\phi^1_i(1)$ are the respective push-forwards of $\phi^*_i(0)$ and $\phi^*_i(1)$ under the same flow of $\ti{X}_i(1)$, we have
\begin{equation}
\label{E:mF fineness control2}
\M(\partial \bar{Q}_i(1)-\partial \phi^1_i(1))\to 0, \text{ uniformly in $\alpha$ as $i\to\infty$}.
\end{equation}

Now we can apply the interpolation result \cite[Theorem 5.1]{Zhou15b} (see also \cite[Theorem 13.1]{MN14}) to $\bar{Q}_i$, which gives that for any $\eta>0$, there exist $l_\eta>0$ and $Q_i: \al(l_\eta)_0\to \C(M)$, such that
\begin{itemize}
\item[(i)] given $x\in \al(l_\eta)_0$, 
\[ \M(\partial Q_i(x)) \leq \M(\partial \bar{Q}_i(x))+\eta/2, \]
so by the same argument as in the proof of point (v) of Lemma \ref{L:interpolation1}, 
\[  \M(Q_i(x)-\bar{Q}_i(x)) \leq \eta/(2c), \]
where $c=\sup |h|$, and hence
\[ \Ac(Q_i(x))\leq \Ac(\bar{Q}_i(x))+\eta; \]

\item[(ii)] $\f(Q_i)\leq \eta$;

\item[(iii)] $\sup\{\F(\partial Q_i(x)-\partial\bar{Q}_i(x)): x\in \al(l_\eta)_0\}<\eta$.
\end{itemize}
When $\eta\to 0$, by (i, iii) and \cite[2.1(20)]{P81} (see also \cite[Lemma 4.1]{MN14}), we have
\[ \lim_{\eta\to 0}\sup\{\mF(\partial Q_i(x), \partial \bar{Q}_i(x)): x\in \al(l_\eta)_0\}=0. \]

Take a sequence $\eta_i\to 0$, and denote $l_i=k_i+l_{\eta_i}+1$, then we construct $\phi_i: I(1, k_i+l_{\eta_i}+1)\to \C(M)$ by defining $\phi_i$ on each $\al(l_{\eta_i}+1)_0$ by
\[ \phi_i(x)= \left\{ \begin{array}{cl}
Q_i(3x) & \text{ for $x\in [0, 1/3]\cap \al(l_{\eta_i}+1)_0$ } \\
\phi_i^1(1) & \text{ otherwise.}
\end{array} \right. \]
The desired properties (a, b, c, d) of $\phi_i$ follow straightforwardly from (\ref{E:mF fineness control1})(\ref{E:mF fineness control2}) and the properties of $Q_j$. Since $\bar{Q}_i$ is obtained from a continuous deformation from $\phi^*_i$, a further interpolation argument shows that $S$ is homotopic to $S^*$, and hence we finish the proof of Claim 2.
\end{proof}

\bibliography{prescribedmc}
\bibliographystyle{plain}

\end{document}